\documentclass[11pt,a4paper]{article}
\usepackage{amssymb,amsfonts,amsmath,amsthm, mathtools,nccmath}
\usepackage[latin1]{inputenc}
\usepackage[english]{babel}
\usepackage[left=2.7cm,right=2.7cm,bottom=2.67cm,top=2.67cm]{geometry}
\usepackage{dsfont}
\usepackage{mathrsfs}
\usepackage{natbib}
\usepackage{color}
\usepackage{graphicx}
\usepackage{multirow}
\usepackage{multicol}
\usepackage{placeins}
\usepackage[normalem]{ulem}

\usepackage[implicit=false]{hyperref}
\allowdisplaybreaks
    
\theoremstyle{plain}
\newtheorem{thm}{Theorem}
\newtheorem{lem}{Lemma}
\newtheorem{corol}{Corollary}[thm]
\newtheorem{coroll}{Corollary}

\renewcommand{\qed}{\hfill \mbox{\raggedright \rule{.07in}{.1in}}}

\newcommand{\vertiii}[1]{{\left\vert\kern-0.25ex\left\vert\kern-0.25ex\left\vert #1 
    \right\vert\kern-0.25ex\right\vert\kern-0.25ex\right\vert}}

\def\R{\mathbb{R}} 
\def\N{\mathbb{N}} 
 
\def\I{\mathcal{I}} 

\DeclarePairedDelimiter\abs{\lvert}{\rvert}
\DeclarePairedDelimiter\norm{\lVert}{\rVert}

\DeclarePairedDelimiter\parent{(}{)}
\DeclarePairedDelimiter\colc{[}{]}
\DeclarePairedDelimiter\chave{\{}{\}}

\DeclareMathOperator*{\E}{E}

\DeclareMathOperator*{\supp}{supp}

\DeclareMathOperator*{\var}{Var}

\newcommand{\ase}{\mathrm{ASE}}

\long\def\sfootnote[#1]#2{\begingroup%
\def\thefootnote{\fnsymbol{footnote}}\footnote[#1]{#2}\endgroup}
\def\bfootnote{\xdef\@thefnmark{}\@footnotetext}

\begin{document}
\pagestyle{myheadings} 
\markboth{}{Matsuoka and Torrent}

\thispagestyle{empty}
{\centering
\Large{\bf Uniform convergence of kernel averages under fixed design with heterogeneous dependent data.} \vspace{.5cm}\\
\normalsize{{\bf 
Danilo H. Matsuoka${}^{\mathrm{a,}}$\sfootnote[1]{Corresponding author. This Version: \today},\let\thefootnote\relax\footnote{\hskip-.3cm$\phantom{s}^\mathrm{a}$Research Group of Applied Microeconomics - Department of Economics, Federal University of Rio Grande.} Hudson da Silva Torrent${}^\mathrm{b}$\let\thefootnote\relax\footnote{\hskip-.3cm$\phantom{s}^\mathrm{b}$Mathematics and Statistics Institute - Universidade Federal do Rio Grande do Sul.
}
 \\
\let\thefootnote\relax\footnote{E-mails: danilomatsuoka@gmail.com (Matsuoka);  hudsontorrent@gmail.com (Torrent)}
\vskip.3cm
}
}
}
\begin{abstract}
	We provide uniform convergence rates for kernel averages on $[0,1]$ under equally-spaced fixed design points of the form $x_{t,T}=t/T,\ t\in\{1,\dotsc, T\},\ T\in\mathbb{N}$. The rates of weak and strong uniform consistency are derived  under strong mixing and moment conditions and do not require stationarity. The analysis exploits the grid structure and thus complements existing random-design results such as those of \cite{hansen} and \cite{kristensen}, which rely on density-based conditioning arguments. The framework accommodates dependent triangular arrays and is particularly relevant for nonparametric methods applied to time series observed on deterministic grids. As an application, we derive uniform convergence rates for the local linear estimator in a nonparametric regression model with time-varying autoregressive errors. The theoretical results are illustrated through Monte Carlo experiments and an empirical application.
\vspace{.2cm}
\end{abstract}

\section{Introduction}\label{sec1}
    
	The uniform consistency of kernel-based estimators is fundamental for inference in nonparametric time series models with dependent data and has been widely investigated under various mixing conditions  \citep{bierens_uni,peligrad_uni,andrews_uni,masry_uni,nze_uni, fan_uni,hansen,kristensen,bosq,kong_uni,li_uni,hirukawa_uni}. In particular, \cite{hansen} established uniform convergence rates for stationary and strongly mixing data over expanding intervals, thereby accommodating kernel functions with both bounded and unbounded support. \cite{kristensen} extended these results to settings where the data may be heterogeneous and parameter dependent. The latter extension is especially important in semiparametric models whose nonparametric component depends on unknown parameters, as in partially linear and single-index models \citep[see][]{li,xia}, and in simulation-based estimation methods \citep[see][]{creel,kristensen_shin}. The former result, in turn, is useful in settings where data may be nonstationary yet strongly mixing, for example, in Markov-Chains that have not been initialized at their stationary distribution \citep{yu6,kim}.  A particularly direct and commonly encountered application of \cite{kristensen} is the local polynomial regression \citep[see][]{wand_jones} with strongly mixing and nonstationary errors.

       The analytical tools employed by \cite{kristensen} and \cite{hansen} are developed under a random design framework in which expectations are expressed through conditioning on the design variable $X_{i,T}$ and  integration with respect to its Lebesgue density  $f_{i,T}$. In our setting, however, the design points $x_{i,T}=i/T,\ i\in{1,\dotsc,T},\ T\in\mathbb{N},$ are deterministic. As a result, the density-based integral representations and conditional expectation arguments used in those previous studies are not directly available in this deterministic setting. Instead, our proofs proceed through deterministic uniform approximations of integrals by finite sums.  Such equally spaced fixed designs are standard in time series analysis, where observations are typically recorded on deterministic grids. They arise naturally in nonparametric time series regression \citep[among others]{robinson,hall_hart,machkouri,vogt}, in time-varying models \citep{dahlhaus,cai}, and when continuous-time processes are sampled at discrete time points \citep{bandi,kristensen2010}.
    
    	Although convergence rates are expected to be of similar order in fixed designs such as $x_{i,T}=i/T$, this setting falls outside the scope of \cite{kristensen} and \cite{hansen}, whose theorems are derived under random design  assumptions in which the design variables admit a density that is absolutely continuous with respect to the Lebesgue measure. Consequently, their results do not immediately extend to the deterministic fixed design setting  considered here  and require further arguments. 
        
        This paper establishes weak and strong uniform convergence rates for kernel averages under fixed design, building the analysis directly on the grid structure. The data are allowed to be strongly mixing, nonstationary, and dependent on a parameter $\gamma$ taking values in a parametric space $\Theta\subseteq\R^m$.  The kernel function is assumed to be compactly supported and Lipschitz.  The theoretical results are subsequently applied to a local linear regression model with time-varying autoregressive errors.
    
	The paper is organized as follows. Section \ref{sec2} develops the main theoretical results. Section \ref{sec3} illustrates their applicability in the context of nonparametric regression  with time-varying autoregressive errors, and further includes both a Monte Carlo investigation and a real-data example using Black Sea mean sea level anomalies.

	\section{General results for kernel averages}\label{sec2}
	This section develops uniform bounds for kernel averages of the form
	\begin{equation}\label{eq1_}
	\hat\Psi(x,\gamma)=T^{-1}\sum_{i=1}^T \epsilon_{i,T}(\gamma) K_h(i/T-x) \parent[\bigg]{\frac{i/T-x}{h}}^j,\quad x\in [0,1],\ \gamma\in\Theta,
	\end{equation}
     where $j\in\N$ is fixed and $\{\epsilon_{t,T}(\gamma):1\leq t\leq T, T >1\}$ denotes a triangular array of random variables defined on $(\Omega,\mathcal{F},P)$, depending on a parameter $\gamma\in\Theta\subseteq \R^m$. For a kernel function $K:\R\to\R$, we denote $K_h(u)\coloneqq K(u/h)/h$ where $h\coloneqq h_T$ is a positive sequence satisfying $h\to 0$ and $Th\to\infty$ as $T\to\infty$. Quantities of the form \eqref{eq1_} are fundamental to time series kernel regression, as they naturally arise in the expressions defining kernel estimators (see \cite{wand_jones,tsybakov}). When $j=0$ we recover the standard kernel average considered in \cite{hansen} and \cite{kristensen}. Including the factor $((i/T-x)/h)^j$ is a convenient generalization, since \eqref{eq1_} is exactly the quantity that appears in local polynomial estimators, making its application to  such settings immediate.

	For each $T\geq1$ and $\gamma\in\Theta$, the $\alpha$-\textit{mixing coefficients} of $\epsilon_{1,T}(\gamma),\dotsc, \epsilon_{T,T}(\gamma)$ are defined by
	\begin{equation*}
		\alpha_{\gamma,T}(j)=\sup_{1\leq k\leq T-j} \sup\{\abs{P(A\cap B)-P(A)P(B)}:B\in\mathcal{F}_{T,1}^k(\gamma),A\in\mathcal{F}_{T,k+j}^T(\gamma)\},
	\end{equation*}
	for any $0\leq j<T$, where $\mathcal{F}_{T,i}^k(\gamma)=\sigma(\epsilon_{l,T}(\gamma):i\leq l\leq k)$. By convention, we set $\alpha_{\gamma,T}(j)=1/4$ for $j\leq 0$ and $\alpha_{\gamma,T}(j)=0$ for $j\geq T$. This definition follows \cite{francq} and \cite{withers}. We say that $\{\epsilon_{i,T}(\gamma): 1\leq i \leq T, T\geq 1\}$ is $\alpha$-\textit{mixing} (or \textit{strongly mixing}) if the sequence
	$$\alpha_\gamma(j)=\sup_{T>j} \alpha_{\gamma,T}(j),\quad 0\leq j<\infty,$$
	satisfies $\alpha_\gamma(j)\to 0$ as $j\to\infty$.
    
	The following assumptions are made throughout this study:
	\begin{enumerate}
		\item[A.1][Strong Mixing] The triangular array $\{\epsilon_{i,T}(\gamma):1\leq i \leq T,T\geq 1\}$ is strongly mixing with mixing coefficients satisfying
		\begin{equation}\label{eq2_}
			\alpha_{\gamma,T}(i)\leq Ai^{-\beta},
		\end{equation}
		for some  finite constants $A>0$ and $\beta>2$ that do not depend on $i,\gamma,T$.
		\item[A.2][Kernel Function] The function $K:\R\to\R$ satisfies $|K(u)|\leq \bar K<\infty$ and $\int_{\R} \abs{K(u)}^s du\leq \bar \mu<\infty$ for some $s>2$. There exist constants $1\leq \Lambda_1,L_1<\infty$ such that $K(u)=0$ for $|u|>L_1$, and $|K(u)-K(u')|\leq \Lambda_1 |u-u'|$ for all $u,u'\in\R$.
        \item[A.3] [Parameter Dependence] For each $T\geq 1$ and $1\leq i\leq T$, there exist a nonnegative random function $\xi_{i,T}(\gamma)$, such that almost surely
        \begin{equation}\label{eqlipz}
            |\epsilon_{i,T}(\gamma')-\epsilon_{i,T}(\gamma)|\leq \xi_{i,T}(\gamma)\|\gamma'-\gamma\|,\qquad \gamma',\gamma\in\Theta: \|\gamma'-\gamma\|\leq h.
        \end{equation}
        Moreover, there exist finite constants $s>2, \lambda\geq 0$, and $\bar C_1, \bar C_2>0$ such that
        \begin{align}\label{eqda1}
            &\sup_{T\geq 1}\,\sup_{1\leq i\leq T}\E(| \epsilon_{i,T}(\gamma)|^s)\leq \bar C_1(1+\|\gamma\|^\lambda),\\\label{eqda2}
            &\sup_{T\geq 1}\,\sup_{1\leq i\leq T}\E(|\xi_{i,T}(\gamma)|^s)\leq \bar C_2(1+\|\gamma\|^\lambda).
        \end{align}
	\end{enumerate}
   
	Following \cite{kristensen} and \cite{hansen}, Assumption A.1 requires the triangular array to be arithmetically $\alpha$-mixing \citep[see Definition 10.2 of][]{ferraty}. The exponent $\beta $ quantifies the rate at which the mixing coefficients decay, with smaller values  corresponding to stronger dependence.

    Assumption A.2  sets standard regularity conditions on the kernel function, including boundedness and integrability of $K$. A.2 encompasses the class of  compactly supported Lipschitz kernels. This includes popular choices such as the Epanechnikov, Biweight, Triweight and Triangular kernels \citep[see Section 2.7 in][]{wand_jones}. 
    Note that, under A.2, both classes of integrals are bounded: $\int_\R |K(u)|^kdu$ for $k\in\{0,\dotsc,s\}$, and $\int_\R |u|^k |K(u)|du$, for $k\in\{0,\dotsc,j\}$.

    Condition \eqref{eqlipz} in Assumption A.3 requires that each mapping $\gamma\mapsto\epsilon_{i,T}(\gamma)$ be locally Lipschitz almost surely, with a random Lipschitz coefficient $\xi_{i,T}(\gamma)$. This assumption is weaker than Assumption A.2 in \cite{kristensen}, which imposes almost sure differentiability and therefore implies almost sure local Lipschitz continuity on $\Theta$. The parameter space $\Theta$ is allowed to be unbounded, and the bounds in \eqref{eqda1}-\eqref{eqda2} ensure that moments   of order up to $s$  remain finite and may grow  with $\|\gamma\|$ at most polynomially. This formulation is aligned with  the growth conditions imposed in Assumptions A.3-A.5 of \cite{kristensen}. If $\Theta$ is compact, the term $1+\|\gamma\|^\lambda$ in \eqref{eqda1}-\eqref{eqda2} is uniformly bounded for all $\lambda\geq0$. Hence, we may set $\lambda=0$ without loss of generality. If the data are parameter independent, then the parameter space is taken to be a trivial singleton, which implies  $m=0$, $\epsilon_{i,T}(\gamma)$ is constant in $\gamma$ so that $ \xi_{i,T}(\gamma)=0$, and the moment bounds in \eqref{eqda1}-\eqref{eqda2} likewise reduce to the case $\lambda=0$. Thus, A.3 collapses to the standard requirement that $\E(| \epsilon_{i,T}|^s)\leq \bar C_1$, uniformly in $i$ and $T$.

	From now on,  
    we use $C>0$ to denote a generic constant which may take different values at different occurrences and is independent of $T, x$, and $\gamma$. The notation ``$\overset{a}{\approx}$'' stands for asymptotic equivalence.

	\subsection{Uniform convergence in probability}\label{sec2.1}

	We now derive a uniform convergence rate in probability for the kernel average in \eqref{eq1_}, adapting Theorem 2 in \cite{hansen} and Theorem 1 in \cite{kristensen} to the fixed design triangular array setting.
	
	\begin{thm}\label{teo2}
		Assume that A.1$-$A.3 hold. Fix $c>0$ and suppose that $s>2$. Define  
        \begin{equation}\label{eq8_}
			\theta=\frac{\beta(s-2)-m(s-1)(1+2c)-2s+1}{\beta s+m(s-1)+1}.
		\end{equation}
       Let  
        $\Theta_T=\{\gamma\in\R^m: \|\gamma\|\leq d_T\}$ with $d_T=T^r$ and $r=\min\{c,(1-\theta)/(2\lambda)\}$.
        Suppose that 
        \begin{equation}\label{eq7_}
			\beta>\frac{m(s-1)(1+2c)+2s-1}{s-2},
		\end{equation}
		and that the bandwidth satisfies
		\begin{equation}\label{eq9_}
			\frac{\ln T}{T^\theta h}=o(1).
		\end{equation}
       If, in addition, the following bounds hold
       \begin{align}\label{eqextrab}
           \sup_{T\geq 1}\,\sup_{1\leq i\leq T}\E(\sup_{\gamma\in\Theta_T}| \epsilon_{i,T}(\gamma)|^s)&\leq  C(1+d_T^\lambda),\\\label{eqextrab2}
             \sup_{T\geq 1}\,\sup_{1\leq i\leq T}\E(\sup_{\gamma\in\Theta_T}| \xi_{i,T}(\gamma)|^s)&\leq  C(1+d_T^\lambda),
       \end{align}
       then 
        \begin{equation*}
            \sup_{\gamma\in \Theta_T}\sup_{x\in [0,1]} \abs{\hat\Psi(x,\gamma)-E\hat\Psi(x,\gamma)}=O_p\left(d_T^\lambda\sqrt{\frac{\ln T}{Th}}\right).
        \end{equation*}
	\end{thm}
	
	Theorem \ref{teo2} establishes a rate of convergence in probability that is uniform in $x\in[0,1]$ and $\gamma\in\Theta_T$. Since the parameter space $\Theta$ may be unbounded, uniformity must be restricted to expanding subsets $\Theta_T$, whose growth rate is determined by $d_T$. In particular, if $\Theta$ is compact, then $\Theta\subseteq \Theta_T$ for all sufficiently large $T$, so uniformity  holds over the entire parameter space $\Theta$. Moreover, if $\Theta$ is compact or data are parameter independent, Assumption A.3 permits any $\lambda\geq 0$. In this case, the optimal choice is $\lambda=0$, which recovers the rate $\sqrt{\ln T/(Th)}$ obtained by \cite{hansen}. The rate $d_T=T^r$ with $r=\min\{c,(1-\theta)/(2\lambda)\}$ guarantees $d_T^\lambda\sqrt{\ln T/(Th)}=o(1)$, while  $c>0$ plays a role when $\lambda=0$. Indeed,  $c>0$  prevents  $d_T$ and the expressions in \eqref{eq8_}-\eqref{eq7_} from becoming undefined. To see this, note that for $r=(1-\theta)/(2\lambda)$, conditions \eqref{eq8_} and \eqref{eq7_} take the forms 
\begin{equation}\label{eqalternative1}
        \theta=\frac{\beta(s-2)-m(s-1)(1+1/\lambda)-2s+1}{\beta s+m(s-1)(1-1/\lambda)+1},
    \end{equation}
     and 
\begin{equation}\label{eqalternative2}
       \beta>\frac{m(s-1)(1+1/\lambda)+2s-1}{s-2},
    \end{equation}
    respectively. These expressions are well defined  if, and only if, $\lambda>0$.

The lower bound in \eqref{eq7_} links the strength of dependence to the dimension $m$ of the parametric space, the constant $c$ governing the expansion rate of $d_T$ and the moment order $s$ specified in Assumption A.3. Since this bound increases with $m$ and $c$ and decreases with $s$, higher-dimensional parameter spaces, faster expansion of $\Theta_T$ or weaker moment conditions (smaller $s$) impose a stronger restriction on $\beta$, requiring a faster decay of mixing coefficients (equivalently, weaker dependence) as specified in \eqref{eq2_}. When Assumption A.3 holds for all $s>0$ and we let $s\to\infty$, the lower bound in \eqref{eq7_} decreases monotonically to $2+m(1+2c)$. Since $m\geq 0$ and $c>0$, the limiting lower bound exceeds 2, so the condition $\beta>2$ is necessary.

By conditions \eqref{eq8_}-\eqref{eq7_}, $\theta\in(0,1)$. In particular, letting $\beta\to\infty$ (e.g., under geometrically $\alpha$-mixing dependence) and $m=0$, the parameter $\theta$ increases monotonically to $1-2/s$, a value strictly below 1 for all $s>2$. Condition \eqref{eq9_}, in turn, requires that the bandwidth $h$ satisfies $T^\theta h/\ln T\to\infty$, so smaller values of $\theta$ imply stronger conditions on $h$. As $\theta\in(0,1)$, condition \eqref{eq9_} strengthens the conventional assumption that $h=o(1)$ and $Th\to\infty$. 

 Although the constant $c>0$ may be chosen arbitrarily, when $\lambda>0$, selecting $c>(1-\theta)/(2\lambda)$ serves no purpose, as it would only tighten the restrictions in \eqref{eq7_}-\eqref{eq9_} without enlarging $\Theta_T$. Hence, one may assume $c\leq (1-\theta)/(2\lambda)$ whenever $\lambda>0$. In particular, for the choice $c=(1-\theta)/(2\lambda)$ conditions \eqref{eq8_} and \eqref{eq7_} reduce to \eqref{eqalternative1} and \eqref{eqalternative2}, respectively.
 
    The control of the kernel average is based on a truncation decomposition. This follows from the identity $\epsilon_{i,T}=\epsilon_{i,T}I(|\epsilon_{i,T}|>\tau_T)+\epsilon_{i,T}I(|\epsilon_{i,T}|\leq\tau_T)$, where $\tau_T$ denotes the truncation level.  A suitable choice of $\tau_T$, together with the uniform moment conditions \eqref{eqextrab}-\eqref{eqextrab2}, ensures uniform control of the non-truncated components via Markov's inequality, while  the truncated components are controlled through the exponential inequality in Lemma \ref{l1}. For the deterministic grid $x_{i,T}=i/T$, the variance term entering this inequality is of a different asymptotic order in the present framework than in the random design case. Consequently, uniform bounds on the number of indices for which the kernel weight is nonzero become essential, making compact support of the kernel particularly convenient.

	\subsection{Almost sure uniform convergence}\label{sec2.2}
    The almost sure counterpart of Theorem~\ref{teo2}  requires the application of the Borel-Cantelli lemma and therefore demands stronger moments bounds and faster decay of the $\alpha$-mixing coefficients (i.e., stronger conditions on $s$ and $\beta$).  In contrast to Theorem 3 of \cite{hansen}, strict stationarity is not assumed.

	\begin{thm}\label{teo3}
		Assume that  A.1-A.3 hold. Fix $c>0$ and suppose that $s>4$. Define  
        \begin{equation}\label{eq11_}
				\theta=\frac{(\beta+1)(s-4)-(s-1)(5+m(1+2c))}{(\beta+1)+(m+1)(s-1)}.
		\end{equation}
        Let  $\Theta_T=\{\gamma\in\R^m: \|\gamma\|\leq d_T\}$ with $d_T=T^r$ and $r=\min\{c,(1-\theta)/(2\lambda)\}$.
        Suppose that 
		\begin{equation}\label{eq10_}
			\beta>\frac{(s-1)(5+m(1+2c))-(s-4)}{s-4}
		\end{equation}
		and  the bandwidth satisfies
		\begin{equation}\label{eqcond}
			\frac{\phi_T}{T^\theta h}=O(1),
		\end{equation}
		where $\phi_T=\ln T(\ln\ln T)^4$. If, in addition, the following bounds hold
        \begin{align}\label{eqfull}
           \sup_{T\geq 1}\sup_{1\leq i\leq T} \E\bigg(\sup_{\gamma\in\Theta_T} |\epsilon_{i,T}(\gamma)|^s\bigg)&\leq C(1+d_T^\lambda),\\\label{eqfull2}
             \sup_{T\geq 1}\,\sup_{1\leq i\leq T}\E\bigg(\sup_{\gamma\in\Theta_T}| \xi_{i,T}(\gamma)|^s\bigg)&\leq  C(1+d_T^\lambda),
        \end{align}  
		then, $$\sup_{\gamma\in\Theta_T}\sup_{x\in [0,1]} \abs{\hat\Psi(x,\gamma)-E\hat\Psi(x,\gamma)}=o_{a.s.}\left(d_T^\lambda\sqrt{\frac{\ln T}{Th}}\right).$$ 
	\end{thm}
	
	To begin with, note that Theorem \ref{teo3} is established under the stricter requirement $s>4$, in contrast to the milder condition $s>2$ used in Theorem \ref{teo2}. Moreover, for $s>4$, the constraints in \eqref{eq11_}-\eqref{eqcond} imposed in Theorem \ref{teo3} are strictly stronger than those in \eqref{eq8_}-\eqref{eq9_} appearing in Theorem \ref{teo2}.  The underlying reason is that  a larger truncation level $\tau_T$ is needed to guarantee the summability of probabilities associated with the non-truncated component. This, in turn, forces stronger restrictions  to control the truncated component via Liebscher-Rio's exponential inequality (Lemma \ref{l1}). This reflects the standard trade-off whereby almost sure convergence involves stronger conditions than convergence in probability.
    
    As in \eqref{eq7_}, the lower bound in \eqref{eq10_} increases with $m$ and $c$, and decreases with $s$. Letting $s\to\infty$, this bound decreases monotonically to $4+m(1+2c)$, implying in particular that  $\beta>4$ is necessary. Under conditions \eqref{eq11_}-\eqref{eq10_}, $\theta$ lies in $(0,1)$. Moreover, when letting $\beta\to\infty$ and setting $m=0$, $\theta$ increases monotonically to $1-4/s$, which remains strictly below 1 since $s>4$. 
    
	\section{Application to a nonparametric regression model}\label{sec3}
	
	Let $Y_{t,T}\in \R$ satisfy, for each $T\geq 1$ and $t\in\{1,\dotsc,T\}$,
	\begin{align}\label{eqap1}
		Y_{t,T}&=g(t/T)+V_{t,T},\\\label{eqap2}
        V_{t,T}&=\phi(t/T)V_{t-1,T}+e_{t,T},
	\end{align}
	where $g(\cdot)$ and $\phi(\cdot)$ are unknown smooth functions on $[0,1]$,  and  $\{{e_{t,T}}\}_{t=1}^T$ are i.i.d. random variables independent of  $V_{0,T}$ for each $T\geq 1$, satisfying $\E(e_{t,T})=0$ and $\E(|e_{t,T}|^s)<C$ for some $s>2$.  Since $V_{t-1,T}$ is measurable with respect to $\sigma(V_{0,T},e_{1,T},\cdots,e_{t-1,T})$, it follows that $\E(e_{t,T}V_{t-1,T})=0$ for all $t\geq 2$ and $T\geq 1$. We also assume $\E(V_{0,T})=0$ and $\E(V_{0,T}^2)<\infty$, for model identification and stability of the process $\{V_{t,T}\}$, respectively. As shown by \cite{kristensen} and \cite{orbe}, under mild conditions $\{V_{t,T}: 1\leq t\leq T, T\geq 1\}$ is $\alpha$-mixing with geometrically mixing rate
of decay \citep[see Definition 10.2 of][]{ferraty}, which implies Assumption A.1\eqref{eq2_} for any $\beta>0$. Since $g(t/T)$ is deterministic, the $\alpha$-mixing coefficients of $\{Y_{t,T}\}$ coincide with those of $\{V_{t,T}\}$, and hence Assumption A.1 holds for $\{Y_{t,T}\}$ whenever it holds for $\{V_{t,T}\}$.

Model \eqref{eqap1}-\eqref{eqap2} can be rewritten as
\begin{equation}\label{eqap3}
    Y_{t,T}=g(t/T)+\phi(t/T)\left(Y_{t-1,T}-g((t-1)/T)\right)+e_{t,T},
\end{equation}
 which shows that $Y_{t,T}$ fluctuates randomly around a deterministic trend $g(t/T)$, reverting toward it at a rate governed by the coefficient $\phi(t/T)$.  

We adopt a two-step semiparametric procedure to estimate  $g(\cdot)$ and $\phi(\cdot)$. 

\emph{Step 1}. We estimate $g$,  using the  \textit{local linear estimator} defined by 
\begin{equation}\label{eqloc}
    \hat g(x)= e^\intercal_1S_{T,x}^{-1}D_{T,x},\qquad \forall x\in[0,1],
\end{equation}
where $e_1=(1,0)^\intercal$ and
	\begin{align}  \label{eq13_}
		S_{T,x}&=\frac{1}{T}\left[
		\begin{array}{cc}
			\sum_{t=1}^T K_h(x_t-x) & \sum_{t=1}^T K_h(x_t-x)(x_t-x)/h  \\ 
			\sum_{t=1}^T K_h(x_t-x)(x_t-x)/h  &  \sum_{t=1}^T K_h(x_t-x)((x_t-x)/h)^2  
		\end{array}
		\right],\\\label{eq14_}
		D_{T,x}&=\frac{1}{T}\left[
		\begin{array}{c}
			\sum_{t=1}^T Y_{t,T} K_h(x_t-x)  \\ 
			\sum_{t=1}^T Y_{t,T} K_h(x_t-x)(x_t-x)/h 
		\end{array}
		\right],
	\end{align}
	where $x_t\coloneqq t/T$. Straightforward algebra shows that $\hat g$ is linear in $Y_{t,T}$:
\begin{equation}\label{eqpesos}
\hat g(x)=\sum_{t=1}^T W_{t,T}(x) Y_{t,T} ,
\end{equation}
where $W_{t,T}(x)=T^{-1}e^\intercal_1 S_{T,x}^{-1}X\parent[\Big]{\tfrac{t/T-x}{h}}K_h(t/T-x)$ for $X(u)=(1,u)^\intercal$. 

    \emph{Step 2}. Given the estimate $\hat g$, compute the residuals $\hat V_{t,T}=Y_{t,T}-\hat g(t/T)$. We then estimate $\phi(\cdot)$ by the \textit{local constant estimator} 
    \begin{equation}\label{eqnada}
        \hat \phi(x)=\frac{\hat \Psi_1(x)}{\hat \Psi_2(x)}\coloneqq  \frac{1/T\sum_{t=2}^T G_v(x_t-x)\hat V_{t,T} \hat  V_{t-1,T}}{1/T\sum_{t=2}^T G_v(x_t-x)\hat  V_{t-1,T}^2},\qquad \forall x\in[0,1],
    \end{equation}
where $G_v(u)\coloneqq G(u/v)/v$, $G:\R\to\R$ is a kernel function and $v$ is a bandwidth satisfying $v=o(1)$, $Tv\to\infty$, following the formulation in Section 3 of \cite{kristensen}.

    The convergence rates of estimators \eqref{eqloc} and \eqref{eqnada} are obtained under additional assumptions:
    \begin{enumerate}
    \item[A.4][Additional Kernel Regularities] The function $K:\R\to\R$ is nonnegative and  symmetric, satisfying $\mu(\{u\in G_x:K(u)>0\})>0$ where $\mu$ is the Lebesgue measure and 
        \begin{equation*}
		G_x=\begin{cases}
	[-1,0],& \text{if } x= 1\\
	[-1,1],          & \text{if } x\in(0,1)\\
	[0,1], & \text{if } x=0
\end{cases}.
\end{equation*}

		\item[A.5][Smoothness Conditions] The functions $g(\cdot)$ and $\phi(\cdot)$ are twice continuously differentiable on $[0,1]$.

    \item[A.6][Autoregressive Part] Model \eqref{eqap2} satisfies the following conditions:
\begin{itemize}
    \item[(i)] There exists  $0<\bar\phi<1$ such that $\max_{u\in[0,1]}|\phi(u)|\leq \bar\phi$;
    \item[(ii)] There exists $s>2$ such that $\sup_{T\geq 1}\sup_{1\leq t\leq T} \E\big(|e_{t,T}|^{s}\big)\leq C$;
    \item[(iii)] For all $1\leq t\leq T$ and all $T\geq 1$, the error $e_{t,T}$ has unit variance and  density $f_e(\cdot)$ satisfying $\int_{-\infty}^\infty |f_e(u)-f_e(u+a)|du\leq C|a|$.
\end{itemize} 
	\end{enumerate}
Assumption A.4  is satisfied by most commonly used compactly supported kernels, such as the Epanechnikov, uniform, triangular, triweight, and cosine kernels. Assumption A.5 is standard in nonparametric kernel regression, and is particularly useful for giving uniform bounds to local linear weights.  According to Proposition 1 of \cite{orbe}, Assumption A.6 ensures that the array $\{V_{t,T}\}$ satisfies Assumption A.1.

The following theorem establishes the uniform convergence rates of $\hat g$ and $\hat \phi$. To avoid boundary effects, the rate for $\hat \phi$ is stated over an interior subset of $[0,1]$. Without loss of generality, we normalize the kernel support to $L_1=1$. The general case $L_1>0$  follows  by a straightforward rescaling of the kernel argument and does not affect any of the asymptotic rates.

    \begin{thm}\label{teo4}
		Let $K$ and $G$ be kernel functions satisfying Assumptions A.2 and A.4. Suppose that Assumptions A.3, A.5 and A.6 hold. Let $b_T=o(1)$ be any sequence such that $b_T/v\to\infty$ and define $\mathcal I_T\coloneqq[b_T,1-b_T]$.  If the bandwidths $h$ and $v$ satisfy \eqref{eq9_} in Theorem \ref{teo2} with $v\overset{a}{\approx} c_v h$ for some $c_v>0$, then
		\begin{equation*}
			\sup_{x\in[0,1]}\abs{\hat g(x)-g(x)}=O(h^2)+O_p\bigg(\sqrt{\frac{\ln(T)}{Th}}\bigg),
		\end{equation*}
        and
        \begin{equation*}
			\sup_{x\in\I_T}\abs{\hat \phi(x)-\phi(x)}=O_p\bigg(h^2+\sqrt{\frac{\ln(T)}{Th}}\bigg).
		\end{equation*}
	\end{thm}
    
\begin{corol}\label{corol1}
Suppose that Assumptions A.2 and A.4-A.6 hold. Let $\phi(t/T)=\phi$ be constant, for all $T\geq 1, 1\leq t\leq T$. Moreover, the map $\phi\mapsto V_{0,T}(\phi)$ satisfies the following conditions on the parameter dependence: there exists $\xi_{0,T}(\phi)\geq 0$ such that almost surely
\begin{equation*}
    |V_{0,T}(\phi')-V_{0,T}(\phi)|\leq \xi_{0,T}(\phi)|\phi'-\phi|,\qquad \phi',\phi\in [-\bar \phi,\bar\phi]: |\phi'-\phi|\leq h
\end{equation*}
and 
\begin{equation*}
    \sup_{T\geq 1}\E\bigg(\sup_{|\phi|\leq \bar\phi}|V_{0,T}(\phi)|^s\bigg)\leq C,\qquad \sup_{T\geq 1}\E\bigg(\sup_{|\phi|\leq \bar\phi}|\xi_{0,T}(\phi)|^s\bigg)\leq C.
\end{equation*}
Then, 
\begin{equation*}
\sup_{\phi\in[-\bar \phi,\bar\phi]}\sup_{x\in[0,1]}\abs{\hat g(x)-g(x)}=O(h^2)+O_p\bigg(\sqrt{\frac{\ln(T)}{Th}}\bigg).
		\end{equation*}
\end{corol}

If, in addition, Assumption A.6(ii) is strengthened by requiring $s>4$ and the bandwidths satisfy \eqref{eqcond} in Theorem \ref{teo3}, the same results hold almost surely. Drawing $V_{0,T}(\phi)$   from the stationary solution
$V_{0,T}(\phi)=\sum_{j=0}^\infty \phi^j e_{-j,T}$

\subsection{Monte Carlo simulations} 

We analyze the finite-sample performance of the estimators presented in this section through  Monte Carlo experiments. The model defined in \eqref{eqap1}-\eqref{eqap2} is simulated with $g(u)=-70 + 134u - 120\tanh(1.1(u-0.56))$, $\phi(u)=0.755  + 0.02 e^{-0.36u} \sin(1.72\pi u + 2.1)$, $e_{t,T}\overset{i.i.d.}{\sim} N(0,\sigma^2_e)$ and $V_{0,T}\sim N(0,\sigma^2_e/(1-[\phi(0)]^2))$. The simulation is replicated $N=2{,}000$ times for each combination of sample size $T\in\{100,300,700\}$ and parameter $\sigma^2_e\in\{1,3\}$.

The performance of the estimators $\hat g$ and $\hat \phi$ is assessed by the mean average squared error (MASE). 
Let $\hat m$ denote an estimator of $m$, and suppose that $N$ replications are available. Then  $\ase_r(\hat m)\coloneqq  T^{-1}\sum_{t=1}^T (\hat m(t/T)-m(t/T))^2$  is computed for each $r\in\{1,\cdots,N\}$. Accordingly, the MASE is defined  as the average ASE  across all replications,  $L(\hat m)\coloneqq N^{-1}\sum_{r=1}^N \ase_r(\hat m)$. Bandwidths are chosen via the \textit{hv}-block cross-validation method \citep{racine} to account for weak dependence, and the Epanechnikov kernel is employed for smoothing.
\begin{table}[ht]
\renewcommand{\arraystretch}{1.15}
\centering
\caption{Results of the Monte Carlo simulations.}\label{tab:sim}
\vskip.2cm
\begin{tabular}{c|cc|cc}
  \hline
&$L(\hat g)$ & $L(\hat\phi)$&$L(\hat g)$ & $L(\hat\phi)$\\ 
  \hline\hline
$T$&\multicolumn{2}{c}{$\sigma_e^2=1$}&\multicolumn{2}{|c}{$\sigma_e^2=3$}\\
  \hline
100&0.816 & 0.051 & 2.139 & 0.043 \\ 
  300&0.360 & 0.009 & 1.012 & 0.008 \\   
  700&0.171 & 0.002 & 0.477 & 0.002 \\ 
   \hline
\end{tabular}
\end{table}

\begin{figure}[t]
\centering
\includegraphics[width=0.7\textwidth]{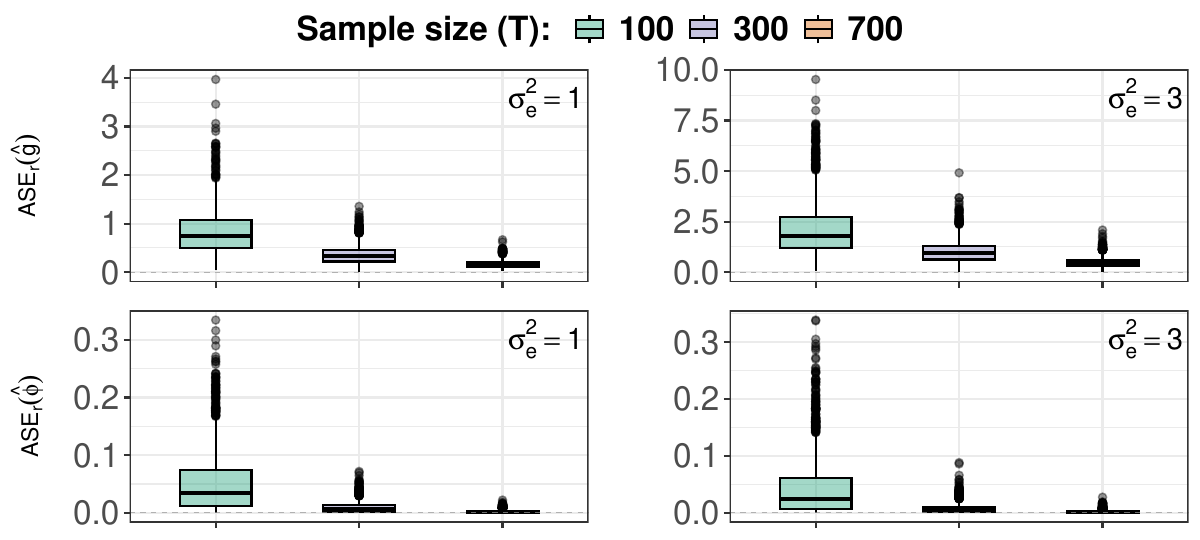}
\caption{Boxplots of the values $\ase_r(\hat g)$ and $\ase_r(\hat \phi)$.}\label{Sim}
\end{figure}

The simulation results are summarized in Table \ref{tab:sim}, which reports MASE values for estimators $\hat g$ and $\hat \phi$. A visual presentation of its finite sample behavior is provided by the boxplots in Figure \ref{Sim}. The results show that MASE values decrease toward zero as $T$ increases, indicating improved estimation accuracy for larger sample sizes. This empirical behavior is consistent with the asymptotic results established in Section \ref{sec2}.

\subsection{Empirical application to sea level anomalies}
We apply our two-step estimation procedure to the monthly mean sea level anomalies (SLA) of the Black Sea. The Black Sea is a semi-enclosed sea in southeastern Europe with limited saltwater exchange with the Mediterranean. Its coastal areas provide favorable conditions for human settlement due to fertile soils and a rich ecosystem (\cite{grinevetsky}), which has drawn attention from numerous scientific studies. The main threats associated with the rising sea level along its coasts are coastal erosion and saltwater intrusion  (\cite{avcsar}). When considering coastal security issues, it is the regional rather than global mean sea levels that are of greatest relevance (\cite{milne,stammer}).

Based on tide gauge and satellite altimetry observations, several studies have reported predominantly positive sloped trends in the Black Sea level since the 1860s \citep{ginzburg,boguslavsky,alpar,avcsar}. The mean sea level exhibited a sharp increase during 1993-1999 (\cite{cazenave}), followed by a more stable upward trend during 1999-2023, showing slight alternating upward and downward movements (\cite{avcsar,wen}). In particular, \citet{wen} documented a rapid increase after 2020.

We use a satellite altimetry dataset from the E.U. Copernicus Marine Service (DOI: 10.48670/moi-00145; accessed on November 8, 2025), corresponding to the Global Ocean Gridded L4 Sea Surface Heights and Derived Variables Reprocessed product. The Black Sea region was defined by the geographical coverage 40\textdegree-48\textdegree N and 26.5\textdegree-42\textdegree E. Daily sea level anomalies were spatially averaged over this region using cosine of latitude weights and temporally aggregated to monthly means. The resulting monthly series was corrected for Glacial Isostatic Adjustment (GIA) using the ICE5G-VM2 model \citep{peltier}, with an estimated regional GIA rate of approximately 0.145 mm/yr based on vertical land motion data from the \href{https://www.atmosp.physics.utoronto.ca/~peltier/data.php}{University of Toronto repository}
(accessed November 8, 2025). The sample covers the period from January 1999 to April 2025 ($T=383$). These data represent gridded sea level anomalies relative to the mean reference period 1993-2012. Annual and semi-annual cycles were removed using the estimators proposed by \citet{vogt}. The resulting Black Sea SLA time series is displayed in Figure \ref{fig:bs} (solid line).

\begin{figure}[ht]
\centering
\includegraphics[width=0.8\textwidth]{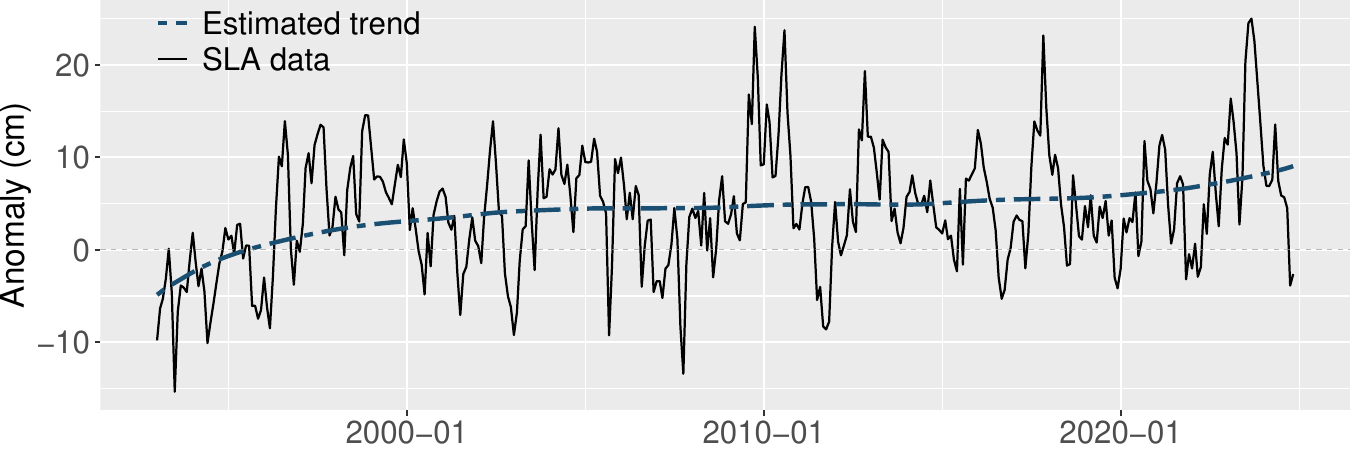}
\vskip -.2cm
\caption{ Monthly sea level anomalies of the Black Sea.}\label{fig:bs}
\end{figure}

Model \eqref{eqap1}-\eqref{eqap2} is fitted to the Black Sea SLA time series using our two-step estimation procedure, yielding estimates of the trend function $g$ and the autoregressive function $\phi$. The autoregressive specification in \eqref{eqap2} is essential to account for the short-term persistence in general SLAs time series. This term captures transient deviations from the long-run mean trajectory, allowing the model to separate the deterministic long-term trend, from short-run dynamics. The chosen bandwidths and kernel function correspond to \textit{hv}-block bandwidths $h_{\text{0}}\approx 0.29$ and $v_{\text{0}}\approx 0.29$, and the Epanechnikov kernel, respectively.  

Figure \ref{fig:bs} also displays the estimated trend function $\hat g$ (blue dashed line). It reveals an overall upward trajectory over the study period, characterized by a deceleration in the early years followed by an acceleration in the latter half. Specifically, the trend shows a sharp increase during 1993-2000, a milder rise until 2020, and a noticeable acceleration in 2020-2025. Such pattern is consistent with the research articles mentioned earlier.

\begin{figure}[ht]
\centering
\includegraphics[width=0.8\textwidth]{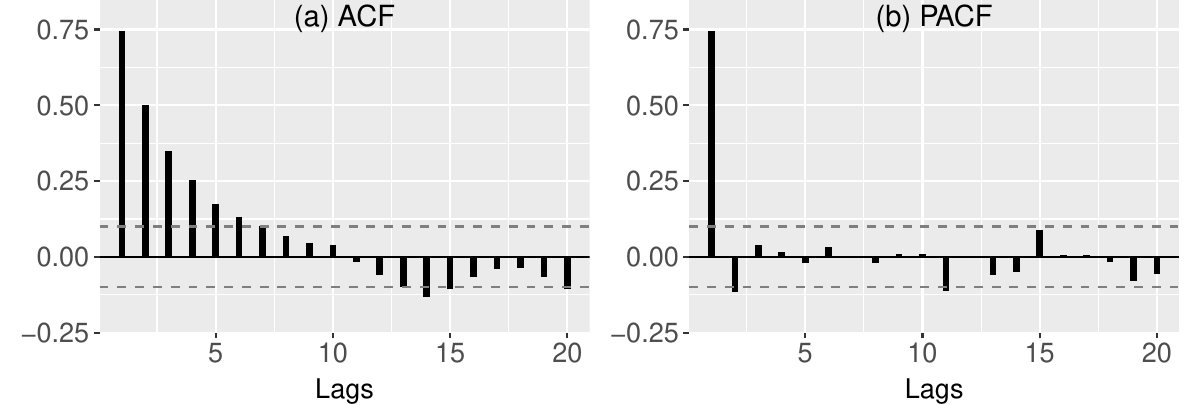}
\vskip -.2cm
\caption{First step residuals diagnostics.}\label{fig:acf1}
\end{figure} 

\begin{table}[ht]
\centering
\caption{Bayesian Information Criterion for fitted ARMA models.}\label{tab:acf}
\vskip.2cm
\begin{tabular}{r|rrrrr}
  \hline
 & MA(0) & MA(1) & MA(2) & MA(3) & MA(4) \\ 
  \hline
AR(0) & 2468.32 & 2256.87 & 2192.68 & 2183.90 & 2178.63 \\ 
  AR(1) & 2160.11 & 2160.76 & 2166.69 & 2172.22 & 2177.93 \\ 
  AR(2) & 2161.04 & 2166.70 & 2172.11 & 2177.69 & 2183.62 \\ 
  AR(3) & 2166.57 & 2172.45 & 2172.57 & 2176.69 & 2182.51 \\ 
  AR(4) & 2172.39 & 2177.70 & 2183.62 & 2189.57 & 2188.58 \\ 
   \hline
\end{tabular}
\end{table}
We briefly analyze the first-step residuals $\hat V_{t,T}=Y_{t,T}-\hat g(t/T)$ using standard diagnostic procedures. Figure \ref{fig:acf1}(b) shows that the  partial autocorrelation function (PACF) drops sharply after lag 1,  while the autocorrelation function (ACF) gradually tails off, as seen in  Figure \ref{fig:acf1}(a). The horizontal dashed lines in Figure \ref{fig:acf1} indicate Bartlett's approximate 95\% confidence limits, $\pm 1.96/\sqrt{T}$, under the null hypothesis of no autocorrelation. An inspection of several ARMA models (Table \ref{tab:acf}) shows that the lowest Bayesian information criterion (BIC) corresponds to the AR(1) specification, with autoregressive parameter of approximately $\phi\approx 0.75$. In particular, provided that  $\phi(\cdot)$ does not vary excessively over time, this finding supports the adequacy of model \eqref{eqap2} for the data.

\begin{figure}[ht]
\centering
\includegraphics[width=0.52\textwidth]{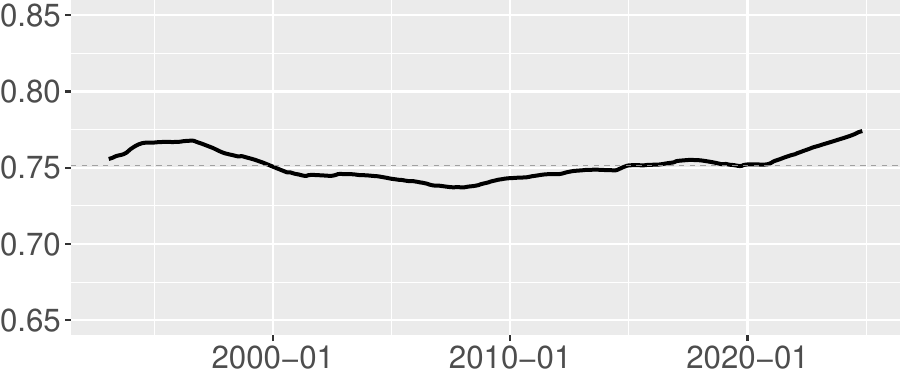}
\vskip -.2cm
\caption{Estimates of the autoregressive function $\hat\phi$.}\label{fig:phi}
\end{figure} 
The estimates $\hat \phi$ obtained from formula \eqref{eqnada} are displayed in Figure \ref{fig:phi}, which suggests a moderate yet stable degree of persistence, with values remaining close to 0.75 throughout the sample period. Together with our previous residual analysis, these results further support the validity of model \eqref{eqap2} for the data.
\begin{figure}[ht]
\centering
\includegraphics[width=.95\textwidth]{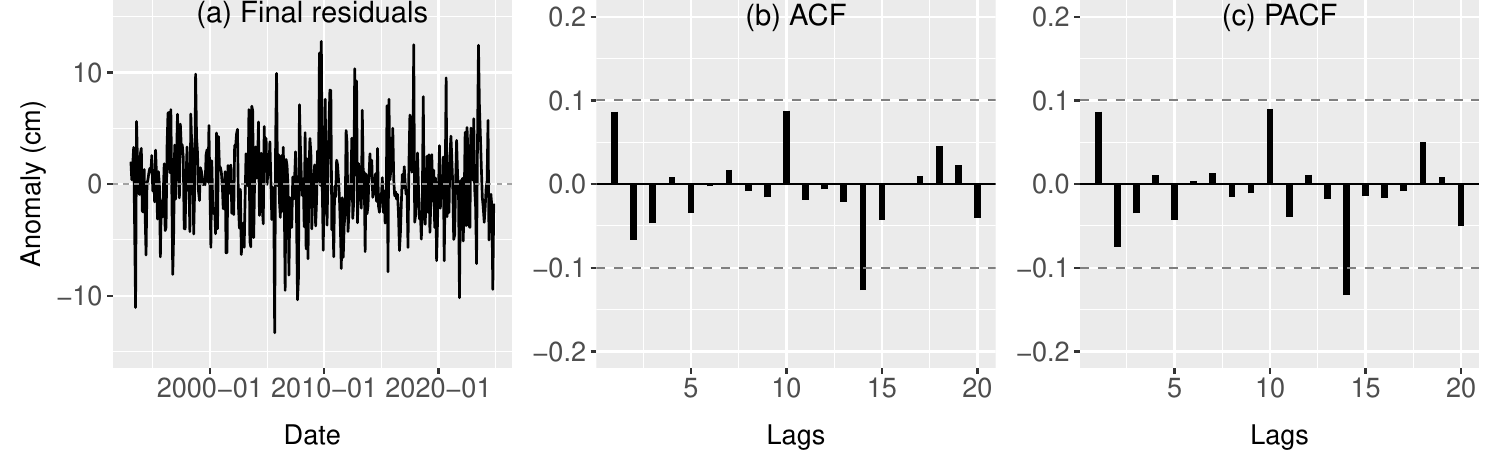}
\vskip -.2cm
\caption{Final residuals diagnostics.}\label{fig:fin_res}
\end{figure} 

Given estimates $\hat g$ and $\hat \phi$, we compute the final residuals $\hat e_{t,T}=\hat V_{t,T}-\hat\phi(t/T)\hat V_{t-1,T}$, shown in Figure \ref{fig:fin_res}(a). The corresponding ACF and PACF, in  Figures \ref{fig:fin_res}(b) and \ref{fig:fin_res}(c), offer no strong evidence of serial dependence. Table \ref{tab:fin} reports the p-values of Ljung-Box tests up to lag 30. All p-values exceed conventional significance levels, implying no sufficient evidence to reject the null of zero autocorrelation in the final residuals. These results for the final residuals provide further support for the overall adequacy of the fitted model.

\begin{table}[ht]
\centering
\caption{P-values of Ljung-Box tests.}\label{tab:fin}
\vskip.2cm
\begin{tabular}{l|ccccccccccc}
  \hline
 Lag& 5 & 8&10 & 13&15 &18& 20 & 23&25&28 &30\\ 
 \hline
  P-value& 0.20 &0.53 &0.42& 0.66& 0.27 &0.42 &0.49 &0.49& 0.53& 0.67 &0.76\\
  \hline
\end{tabular}
\end{table}

	\section{Conclusion}\label{sec4}

    We have established weak and strong uniform convergence rates for general kernel averages under strong mixing conditions in a fixed-design setting with equally spaced grid points $x_{t,T}=t/T$. The analysis departs from the classical random-design framework by developing tools tailored to the deterministic grid structure. This allows us to obtain fixed-design counterparts to the results of \cite{hansen} and \cite{kristensen}, without relying on density-based conditioning arguments.
    The rates are uniform over both the design space and expanding parameter sets, accommodating parameter-dependent triangular arrays. Both weak and almost sure uniform convergence are established without imposing stationarity. However, the almost sure result requires stronger moment and mixing conditions.

    As an illustration, we applied the general theory to local linear regression with time-varying autoregressive errors, derived uniform rates for the estimators, and complemented the theoretical analysis with Monte Carlo experiments and a real-data application.
    
    Overall, the results provide a theoretical foundation for kernel-based inference in nonstationary time-series models under deterministic designs, a setting frequently encountered in practice.

	\section*{Funding Statement}	
	This work was supported by CAPES Foundation (Grant No. 140650/2016-4), Ministry of Education, Brazil.
	
	\section*{Conflict of Interest Statement}
	The authors declare that there are no conflicts of interest. 

    \section*{Data Availability Statement}
    The data used in this study are publicly available from the Copernicus Marine Service (https://doi.org/10.48670/moi-00145). The processed data underlying the empirical analysis can be reproduced from this source following the procedures described in the article.

	\bibliographystyle{apalike}
    \bibliography{bib}

   \section*{Appendix A: Proofs}

We use the following notations: $[k] := \{1,2,\dots,k\}$ for any $k\in\N$, $\sup_x\coloneqq \sup_{x\in [0,1]}$, $\sup_\gamma\coloneqq \sup_{\gamma\in\Theta_T}$, $\sup_i\coloneqq\sup_{1\leq i\leq T}$, $\sup_T\coloneqq \sup_{T\geq 1}$, $\sup_{T,i}\coloneqq \sup_T\sup_i$  and $\sup_{x,\gamma}\coloneqq \sup_x\sup_\gamma$.
	\textbf{Proof of Theorem \ref{teo2}}.
	For brevity, write $k_{i,T}(x)\coloneqq K((i/T-x)/h)$ and $\pi_{i,j,T}(x)\coloneqq((i/T-x)/h)^j$. 
	Decompose
	\begin{align*}
		\hat\Psi(x,\gamma)&=\frac{1}{Th }\sum_{i=1}^T \epsilon_{i,T}(\gamma)k_{i,T}(x)\pi_{i,j,T}(x)I(\abs{\epsilon_{i,T}(\gamma)}> \tau_T)\\
        &\quad +\frac{1}{Th }\sum_{i=1}^T \epsilon_{i,T}(\gamma)k_{i,T}(x)\pi_{i,j,T}(x)I(\abs{\epsilon_{i,T}(\gamma)}\leq \tau_T)\\
		&\coloneqq R_{1,T}(x,\gamma)+R_{2,T}(x,\gamma),
	\end{align*}
	where $I(\cdot)$ is the indicator function and  $\tau_T\coloneqq (a_T h)^{-1/(s-1)}$ with $a_T\coloneqq (\ln T/(Th))^{1/2}$.
    The proof proceeds by controlling the contribution of $R_{1,T}(x,\gamma)$, and then applying an exponential inequality to the truncated term $R_{2,T}(x,\gamma)$ on a suitable grid in $[0,1]\times\Theta_T$.

	We start by focusing on $R_{1,T}$. Denote $Y_{i,T}\coloneqq \sup_\gamma\abs{\epsilon_{i,T}(\gamma)}$. By H\"older's  and Markov's  inequalities, and condition \eqref{eqextrab}, we have uniformly in $i\in[T]$,
	\begin{align}\notag
\E(Y_{i,T}I(Y_{i,T}>\tau_T))&\leq[\E(Y_{i,T}^s)]^{\tfrac{1}{s}} [P(Y_{i,T}>\tau_T)]^{1-\tfrac{1}{s}}\leq [\E(Y_{i,T}^s)]^{\tfrac{1}{s}} \colc[\bigg]{\frac{\E(Y_{i,T}^s)}{\tau_T^s}}^{1-\tfrac{1}{s}}\\\label{eqa1}
        &= \E(Y_{i,T}^s)\tau_T^{1-s}\leq  C (1+d_T^\lambda)\tau_T^{1-s}\leq  C d_T^\lambda\tau_T^{1-s}.
	\end{align}
	On the other hand,  by Lemma \ref{lext}, there exists $C_j>0$ depending only on $j$ such that
\begin{equation}\label{eqponte}
    \sup_{x,\gamma} |R_{1,T}(x,\gamma)|\leq \frac{C_j}{Th}\sum_{i=1}^T Y_{i,T}I(\abs{Y_{i,T}}> \tau_T).
\end{equation}
Thus, applying expectations in \eqref{eqponte} and using inequality \eqref{eqa1}, there exists $c_1>0$ such that
	\begin{align}\label{eqa2}
		\E\Big(\sup_{x,\gamma}\abs{ R_{1,T}(x,\gamma)}\Big)&
        \leq \frac{C_j}{Th}\sum_{i=1}^T \E\big(Y_{i,T}I(\abs{Y_{i,T}}> \tau_T)\big)\leq \frac{c_1}{h}d_T^\lambda\tau_T^{1-s}=c_1d_T^\lambda a_T.
	\end{align}
    Therefore, by Markov's inequality, for all $\delta>0$, taking $C_\delta=2c_1/\delta$ we have
	\begin{align}\notag
		P\parent[\Big]{\sup_{x,\gamma}\abs{R_{1,T}(x,\gamma) -\E R_{1,T}(x,\gamma)}>C_\delta d_T^\lambda a_T}&\leq \frac{2}{C_\delta d_T^\lambda a_T}\E\Big(\sup_{x,\gamma}\abs{ R_{1,T}(x,\gamma)}\Big)\leq \delta,
	\end{align}
    which shows that
\begin{equation}\label{eqcentral}
		\sup_{x,\gamma}\abs{R_{1,T}(x,\gamma) -\E R_{1,T}(x,\gamma)}=O_p(d_T^\lambda a_T).
	\end{equation}
	Thus,  truncating $\epsilon_{i,T}$ at $\tau_T$ incurs an $O_p(d_T^\lambda a_T)$ error uniformly in $x\in[0,1]$ and $\gamma\in\Theta_T$.

	Now, before bounding $R_{2,T}^{(a)}$ we give some useful results. Cover the set $A=[0,1]\times \Theta_T$ with $N=\lceil d_T^m/(a_T h)^{1+m}\rceil$ rectangles of the form
$A_j=\{(x,\gamma):|x-x_j|\leq a_T h, \|\gamma-\gamma_j\|\leq a_T h\}$,  where the centers $(x_l,\gamma_l)$ are chosen in $A$, so that
$[0,1]\times\Theta_T\subseteq\bigcup_{j=1}^N A_j$. Let  $\epsilon_{i,T}^*(\gamma)\coloneqq \epsilon_{i,T}(\gamma)I(|\epsilon_{i,T}(\gamma)|\leq \tau_T)$ and $k^*_{i,T}(x)\coloneqq K^*((i/T-x)/h)$ where $\quad K^*(x)=\Lambda_1 I(\abs{x}\leq 2L_1)$. 
Define
	\begin{align*}
		\tilde \Psi_j(x,\gamma)&=(Th)^{-1}\sum_{i=1}^T \abs{k^*_{i,T}(x)\pi_{i,j,T}(x)\epsilon^*_{i,T}(\gamma)},\\
        \underline \Psi_j(x,\gamma)&=(Th)^{-1}\sum_{i=1}^T \abs{k_{i,T}(x) \pi_{i,j,T}(x)} \xi_{i,T}(\gamma),
	\end{align*}
	where $\xi_{i,T}(\gamma)$ is introduced in Assumption A.3. From Assumption A.3, Lemma \ref{l2}-\ref{l3} with $g(u)=|u|$, and H\"older's inequality, it follows that 
	\begin{align}\notag
		\E\abs[\big]{\tilde \Psi_k(x,\gamma)} &\leq \frac{C(1+d_T^\lambda)^{\tfrac{1}{s}}}{Th}\sum_{i=1}^T \abs{k^*_{i,T}(x)\pi_{i,k,T}(x)}\\\label{eqa4}
        &\quad =\frac{Cd_T^{\tfrac{\lambda}{s}}}{h}\bigg\{h\int_R |K^*(w)||w|^kdw+O(1/T)\bigg\}=O\big(d_T^{\lambda/s}\big),\\\notag
        \E\abs[\big]{\underline \Psi_k(x,\gamma)} &\leq \frac{C(1+d_T^\lambda)^{\tfrac{1}{s}}}{Th}\sum_{i=1}^T \abs{k_{i,T}(x)\pi_{i,k,T}(x)}\\\label{eqa4.2}
        &\quad =\frac{Cd_T^{\tfrac{\lambda}{s}}}{h}\bigg\{h\int_R |K(w)||w|^kdw+O(1/T)\bigg\}=O\big(d_T^{\lambda/s}\big)
    \end{align}
for any $k\leq j$, and uniformly in  $x\in[0,1]$ and $\gamma\in\Theta_T$, since both integrals which appear above  are $O(1)$.

    Fix $l\in\{1,\dots,N\}$ and $(x,\gamma)\in A_l$. Then
\begin{align*}
    |R_{2,T}(x,\gamma)-R_{2,T}(x_l,\gamma_l)|&\leq \frac{1}{Th}\sum_{i=1}^T \abs{\epsilon^*_{i,T}(\gamma)}|\pi_{i,j,T}(x)| \abs{k_{i,T}(x)-k_{i,T}(x_l)}\\
    &\qquad + \frac{1}{Th}\sum_{i=1}^T \abs{k_{i,T}(x_l)}|\pi_{i,j,T}(x_l)| \abs{\epsilon^*_{i,T}(\gamma)-\epsilon^*_{i,T}(\gamma_l)}\\
    &\qquad + \frac{1}{Th}\sum_{i=1}^T |\epsilon^*_{i,T}(\gamma)||k_{i,T}(x_l)||\pi_{i,j,T}(x)-\pi_{i,j,T}(x_l)|\\
    &\coloneqq A_{T}+B_{j,T}+C_{T},
\end{align*}
    with a slight abuse of notation for brevity's sake.  We claim that 
     \begin{align}\label{eqclaim1}
         A_{T}&\leq  2^{j-1} a_T \big(\tilde \Psi_j(x_l,\gamma)+\tilde \Psi_0(x_l,\gamma)\big),\\\label{eqclaim3}
         C_T&\leq j2^j a_T(\tilde\Psi_{d}(x_l,\gamma)+\tilde\Psi_0(x_l,\gamma))
     \end{align}
     where $d=\max\{0,j-1\}$, and for all $k\leq j$
     \begin{equation}\label{eqclaim2}
          B_{k,T}\leq a_T h \underline \Psi_k(x_l,\gamma_l)+U_{l,T}
     \end{equation}
     with 
     \begin{equation}\label{eqdanremen2}
      U_{l,T}\coloneqq   \frac{2}{Th}\sum_{i=1}^T \abs{k_{i,T}^*(x_l)}|\pi_{i,k,T}(x_l)| \sup_{\gamma}|\epsilon_{i,T}(\gamma)|I\Big(\sup_{\gamma}|\epsilon_{i,T}(\gamma)|>\tau_T\Big),
     \end{equation}
     which satisfies 
     \begin{equation}\label{eqdanremen3}
         \max_{1\leq l\leq N}U_{l,T}=O_p(d_T^\lambda a_T) \ \text{ and } \ \E\bigg(\max_{1\leq l\leq N} U_{l,T}\bigg)\leq Cd_T^\lambda a_T.
     \end{equation}
    The proof of \eqref{eqclaim1}-\eqref{eqdanremen3} is postponed until the arguments are completed. Hence, using the bounds \eqref{eqa4}-\eqref{eqclaim2}, we obtain for all sufficiently large $T$,
	\begin{align}\notag
		&\abs{R_{2,T}(x,\gamma)-\E R_{2,T}(x,\gamma)}
        \leq\abs{R_{2,T}(x_l,\gamma_l)-\E R_{2,T}(x_l,\gamma_l)}\\\notag
			&\qquad + \abs{R_{2,T}(x,\gamma)-R_{2,T}(x_l,\gamma_l)}+\E\abs{R_{2,T}(x,\gamma)-R_{2,T}(x_l,\gamma_l)}\\ \notag
		&\leq  \abs{R_{2,T}(x_l,\gamma_l)-\E R_{2,T}(x_l,\gamma_l)}
        +(A_T+\E A_T)+(B_{j,T}+\E B_{j,T})+(C_T+\E C_T)\\ \notag
        &\leq  \abs{R_{2,T}(x_l,\gamma_l)-\E R_{2,T}(x_l,\gamma_l)}
        +Ca_T\sum \nolimits_{k\in\{0, d, j\}}\big(\tilde \Psi_k(x_l,\gamma)+\E\tilde \Psi_k(x_l,\gamma)\big)\\ \notag
        &\qquad + a_T h \big(\underline \Psi_j(x_l,\gamma_l)+\E\underline \Psi_j(x_l,\gamma_l)\big)+U_{l,T}+C d_T^\lambda a_T\\\notag
		  &\leq  \abs{R_{2,T}(x_l,\gamma_l)-\E R_{2,T}(x_l,\gamma_l)}
        +Ca_T\sum \nolimits_{k\in\{0, d, j\}}\big|\tilde \Psi_k(x_l,\gamma)-\E\tilde \Psi_k(x_l,\gamma)\big|\\ \label{eqtenta}
        &\qquad + a_T h \big|\underline \Psi_j(x_l,\gamma_l)-\E\underline \Psi_j(x_l,\gamma_l)\big|+U_{l,T}+C d_T^\lambda a_T.
	\end{align}
    By similar arguments used for $R_{2,T}$ in \eqref{eqtenta}, we obtain for any $k\leq j$
    \begin{align}\notag
        \big|\tilde \Psi_k(x_l,\gamma)-\E \tilde \Psi_k(x_l,\gamma)\big|&\leq \big|\tilde \Psi_k(x_l,\gamma_l)-\E \tilde \Psi_k(x_l,\gamma_l)\big|\\\label{eqjahre}
        &\quad +a_T h \big|\underline \Psi_k(x_l,\gamma_l)-\E\underline \Psi_k(x_l,\gamma_l)\big|+U_{l,T}+C d_T^\lambda a_T.
    \end{align}
    Thus, combining \eqref{eqclaim1}-\eqref{eqjahre}, and using $Ca_T\leq 1$ for $T$ large enough, it follows that 
\begin{align}\notag
    &\sup_{(x,\gamma)\in A_l}\abs{R_{2,T}(x,\gamma)-\E R_{2,T}(x,\gamma)}
        \leq\abs{R_{2,T}(x_l,\gamma_l)-\E R_{2,T}(x_l,\gamma_l)}\\\notag
        &\qquad + \sum\nolimits_{k\in\{0,j,d\}}\Big\{\big|\tilde \Psi_k(x_l,\gamma_l)-\E \tilde \Psi_k(x_l,\gamma_l)\big|+a_T h \big|\underline \Psi_k(x_l,\gamma_l)-\E\underline \Psi_k(x_l,\gamma_l)\big|\Big\}\\\notag
        &\qquad +4U_{l,T}+Cd_T^\lambda a_T\\\notag
        &\leq \underbrace{\abs{R_{2,T}(x_l,\gamma_l)-\E R_{2,T}(x_l,\gamma_l)}}_{\coloneqq L_{l,T}^R}+\sum\nolimits_{k\in\{0,j,d\}}\underbrace{\big|\tilde \Psi_k(x_l,\gamma_l)-\E \tilde \Psi_k(x_l,\gamma_l)\big|}_{\coloneqq L^{\tilde\Psi,k}_{l,T}}\\\notag
        &\qquad+\sum\nolimits_{k\in\{0,j,d\}} \underbrace{a_T h\big|\underline \Psi_k(x_l,\gamma_l)-\E\underline \Psi_k(x_l,\gamma_l)\big|}_{Q_{l,T}^{\xi,k}}+4U_{l,T}+Cd_T^\lambda a_T\\\label{eqa6}
        &\coloneqq L_{l,T}^R+\sum\nolimits_{k\in\{0,j,d\}} L^{\tilde\Psi,k}_{l,T}+  \sum\nolimits_{k\in\{0,j,d\}} Q_{l,T}^{\xi,k}+4U_{l,T}+Cd_T^\lambda a_T.
\end{align}
    Note that the centered term $Q_{l,T}^{\xi,k}=a_T h|\underline \Psi_k(x_l,\gamma_l)-\E\underline \Psi_k(x_l,\gamma_l)|$ need not be bounded, so Liebscher-Rio's inequality (Lemma \ref{l1}) cannot be applied directly. To circumvent this, decompose $\underline \Psi_k=\underline \Psi_k^{(a)}+\underline \Psi_k^{(b)}$ where
\begin{align*}
\underline \Psi_k^{(a)}(x_l,\gamma_l)&=(Th)^{-1}\sum\nolimits_{i=1}^T \abs{k_{i,T}(x_l) \pi_{i,k,T}(x_l)} \xi_{i,T}(\gamma_l)I(\xi_{i,T}(\gamma_l)>\tau_T)\\
\underline \Psi_k^{(b)}(x_l,\gamma_l)&=(Th)^{-1}\sum\nolimits_{i=1}^T \abs{k_{i,T}(x_l) \pi_{i,k,T}(x_l)} \xi_{i,T}(\gamma_l)I(\xi_{i,T}(\gamma_l)\leq \tau_T).
\end{align*}
Then, by the triangle inequality
\begin{align}\label{equadanf}
   Q_{l,T}^{\xi,k}\leq a_T h\big|\underline\Psi_k^{(a)}-\E\underline\Psi_k^{(a)}\big|
+
\big|\underline\Psi_k^{(b)}-\E\underline\Psi_k^{(b)}\big| \coloneqq M_{l,T}^{\xi,k}+L_{l,T}^{\xi,k},
\end{align}
since $a_Th\leq 1$ for $T$ large enough.

Inequalities \eqref{eqa6}-\eqref{equadanf} give
\begin{equation}\label{eqnewres}
    \sup_{(x,\gamma)\in A_l}\big|R_{2,T}(x,\gamma)-\E R_{2,T}(x,\gamma)\big|
\leq L_{l,T}+M_{l,T}+4U_{l,T}+C d_T^\lambda a_T,
\end{equation}
where 
\begin{align*}
    L_{l,T}&\coloneqq L_{l,T}^R+\sum_{k\in\{0,d,j\}} L^{\tilde\Psi,k}_{l,T}+\sum_{k\in\{0,d,j\}}  L^{\xi,k}_{l,T} \quad \text{ and } \quad M_{l,T}\coloneqq \sum_{k\in\{0,d,j\}}  M_{l,T}^{\xi,k}.
\end{align*}
$L_{l,T}$ consists of centered kernel averages with bounded summands and is
therefore controlled via Liebscher-Rio's inequality, whereas  $M_{l,T}$ collects
the remaining tail terms and is controlled via Markov's inequality.

	Denote $e_T(x,\gamma)\coloneqq\abs{R_{2,T}(x,\gamma)-ER_{2,T}(x,\gamma)}$.  Since $A=[0,1]\times\Theta_T \subseteq \bigcup_{l=1}^NA_l$ and inequality \eqref{eqnewres} holds, the following relations of measurable sets are valid for all $C>0$:  
    {\small
    \begin{align*}
        \chave[\bigg]{\sup_{(x,\gamma)\in A} e_T(x,\gamma)>C}&\subseteq \chave[\bigg]{\sup_{(x,\gamma)\in \cup A_l} e_T(x,\gamma)>C}=\bigg\{\max_{1\leq l\leq N}\sup_{(x,\gamma)\in A_l} e_T(x,\gamma)>C\bigg\}\\
        &\subseteq\bigg\{\max_{1\leq l\leq N}  L_{l,T}+M_{l,T}+4U_{l,T}>C\bigg\}\\
        &\subseteq\bigg\{\max_{1\leq l\leq N} L_{l,T}>C/3\bigg\}\bigcup \bigg\{\max_{1\leq l\leq N} M_{l,T}>C/3\bigg\}\bigcup \bigg\{\max_{1\leq l\leq N} U_{l,T}>C/12\bigg\}\\
        &\subseteq \bigg(\bigcup_{l=1}^N\chave[\bigg]{ L_{l,T}> C/3}\bigg)\bigcup \bigg\{\max_{1\leq l\leq N} M_{l,T}>C/3\bigg\}\bigcup \bigg\{\max_{1\leq l\leq N} U_{l,T}>C/12\bigg\}.
    \end{align*}
    }
	Consequently, by the monotonicity and subadditivity of the measure,
	\begin{align}\notag
		P\bigg(\sup_{(x,\gamma)\in A}& \abs{R_{2,T}(x,\gamma)-ER_{2,T}(x,\gamma)}>21C d_T^\lambda a_T \bigg)\\\notag
        &\leq P\parent[\bigg]{\max_{1\leq l\leq N} M_{l,T}>6Cd_T^{\lambda}a_T} +P\parent[\bigg]{\max_{1\leq l\leq N} U_{l,T}>(21/12)Cd_T^{\lambda}a_T}\\\notag
      &\qquad +\sum_{l=1}^N P\parent[\bigg]{ L_{l,T}>7Cd_T^{\lambda}a_T} \\\notag
        &\leq \underbrace{P\parent[\bigg]{\max_{1\leq l\leq N} \max_{0\leq k\leq j} M_{l,T}^{\xi,k}>6Cd_T^{\lambda}a_T}}_{\coloneqq T^{(M)}_0}+\underbrace{P\parent[\bigg]{\max_{1\leq l\leq N} U_{l,T}>Cd_T^{\lambda}a_T}}_{\coloneqq T_0^{(U)}}\\\notag
        &\qquad + \underbrace{N \max_{1\leq l\leq N} P(L_{l,T}^R>Cd_T^\lambda a_T)}_{\coloneqq T_1}+\underbrace{\sum_{k\in\{0,d,j\}}N \max_{1\leq l\leq N} P(L_{l,T}^{\tilde\Psi,k}>Cd_T^\lambda a_T)}_{\coloneqq \sum\nolimits_{q=2}^4 T_q}\\\notag
        &\qquad +\underbrace{\sum_{k\in\{0,d,j\}}N \max_{1\leq l\leq N} P(L_{l,T}^{\xi,k}>Cd_T^\lambda a_T)}_{\coloneqq \sum\nolimits_{q=5}^7 T_q}\\\label{eqa7}
		&\coloneqq  T_0^{(M)}+T_0^{(U)}+\sum\nolimits_{q=1}^7 T_q,
	\end{align}
	for sufficiently large $T$. From \eqref{eqdanremen3}, we already have $\max_{1\leq l\leq N}U_{l,T}=O_p(d_T^\lambda a_T)$, so $T_0^{(U)}$ is controlled. Next, we handle the tail probability $T^{(M)}_0$. In the same fashion as we did in \eqref{eqa1}, condition A.3\eqref{eqda2} implies $ \E(\sup_\gamma\xi_{i,T}(\gamma)I(\sup_\gamma\xi_{i,T}(\gamma)>\tau_T))\leq C d_T^\lambda \tau_T^{1-s}$. Therefore, from Lemma  \ref{lext} and condition \eqref{eqextrab2}, 
    \begin{align*}
        \E \Big(\max_{1\leq l\leq N}\max_{0\leq k\leq j}\Psi_k^{(a)}(x_l,\gamma_l)\Big)&\leq \frac{C_j}{Th}\sum_{i=1}^T \E(\sup_\gamma\xi_{i,T}(\gamma)I(\sup_\gamma\xi_{i,T}(\gamma)>\tau_T))\\
        &\leq  C_a\frac{d_T^\lambda \tau_T^{1-s}}{h}=C_ad_T^\lambda a_T,
    \end{align*}
	for some $C_a>0$, so using Markov's inequality,
\begin{align*}
    T_0^{(M)}&=P\parent[\bigg]{\max_{1\leq l\leq N}\max_{0\leq k\leq j} M^{\xi,k}_{l,T}>7Cd_T^{\lambda}a_T }\leq \E\Big(\max_{1\leq l\leq N}\max_{0\leq k\leq j}M^{\xi,k}_{l,T}\Big)\frac{1}{7Cd_T^\lambda a_T}\\
    &\leq 2a_T h\E\Big(\max_{1\leq l\leq N}\max_{0\leq k\leq j}\Psi_k^{(a)}(x_l,\gamma_l)\Big)\frac{1}{2Cd_T^\lambda a_T}\leq \frac{C_a a_T h}{C}=o(1).
\end{align*}
    
    We can thus focus only on the truncated terms $T_1,\dotsc,T_7$. 
	We proceed by bounding $T_1$ using Lemma \ref{l1} (Liebcher-Rio), the remaining terms are controlled analogously.
    Set 
\begin{equation*}
    Z_{i,T}(x,\gamma)\coloneqq\epsilon^*_{i,T}(\gamma)\pi_{i,j,T}(x)k_{i,T}(x)-\E\big(\epsilon^*_{i,T}(\gamma)\pi_{i,j,T}(x)k_{i,T}(x)\big),
\end{equation*}
which satisfies $\abs{Z_{i,T}(x,\gamma)}\leq 2C_j\tau_T \coloneqq b_T$ by Lemma  \ref{lext} and $\E Z_{i,T}(x,\gamma)=0$.
Since $K$ is supported on $[-L_1,L_1]$, we have
\begin{equation*}
    R_{2,T}(x,\gamma)-ER_{2,T}(x,\gamma)=\frac{1}{Th}\sum_{i=1}^T Z_{i,T}(x,\gamma)=\frac{1}{Th}\sum_{i\in J_x(L_1)} Z_{i,T}(x,\gamma),
\end{equation*}
where 
\begin{equation*}
    J_x(L_1)\coloneqq\Big\{i\in\{1,\dots,T\}:\Big|\frac{i/T-x}{h}\Big|\leq L_1\Big\},
\end{equation*}
whose cardinality, denoted as $n_{T}(x)\coloneqq \#J_x(L_1)$, satisfies $\sup_x n_{T}(x)\leq 4 L_1 Th$ (Lemma \ref{l2.1}).
     By Lemma \ref{teo1},  there exists $C_1>0$ such that  uniformly in $(x,\gamma)\in A$,
    \begin{equation*}
       \sigma^2_{T,m_T}(x,\gamma)\coloneqq \sup_{0\leq j \leq T-1}E\Bigg[\bigg(\sum_{i=1}^{m_T} Z_{j+i,T}(x,\gamma)\bigg)^2\Bigg]\leq C_1 d_T^{\lambda} m_T.
    \end{equation*}
	Set $m_T=(a_T\tau_T)^{-1} $ and  $\epsilon=Md_T^\lambda a_TTh$, where $M>0$ is any constant such that $M>C_1$. Observing that $m_T\to\infty$, $m_T\leq \sqrt{T}\leq T$ and $\tfrac{m_Tb_T}{d_T^\lambda a_T Th}=\tfrac{2\bar K}{d_T^\lambda \ln T}=o(1)$, the conditions of Lemma \ref{l1} (Liebscher-Rio's lemma) hold for $T$ large enough.
    Thus, for all $(x,\gamma)\in A$ and all $T$ sufficiently large, 
	\begin{align}\notag
		P(\abs{R_{2,T}(x,\gamma)-&ER_{2,T}(x,\gamma)}>M d_T^\lambda a_T)=		P\parent[\bigg]{\abs[\bigg]{\sum_{i\in J_x(L_1)} Z_{i,T}(x,\gamma)}>Md_T^\lambda a_T Th}\\\notag
		&\leq  4\exp\parent[\bigg]{-\frac{(Md_T^\lambda a_TTh )^2}{64\sigma^2_{T,m_T}n_T(x)/m_T+(Md_T^\lambda a_TTh ) b_T m_T8/3}}+4\alpha_T(m_T)\frac{n_T(x)}{m_T}\\\notag
		&\leq  4\exp\parent[\bigg]{-\frac{(Md_T^\lambda a_TTh )^2}{254 L_1C_1 d_T^\lambda Th +6M d_T^\lambda \bar K Th }}+16L_1(Am_T^{-\beta})\frac{Th}{m_T}\\\notag
		&\leq  4\exp\colc[\bigg]{-\frac{M d_T^\lambda}{254L_1+6\bar K}\ln T}+16L_1Am_T^{-1-\beta}Th\\\label{eqa8}
		&=  4T^{-Md_T^\lambda /(254L_1+6\bar K)}+16L_1ATh(a_T\tau_T)^{1+\beta}.
	\end{align}
	The same bound holds for $T_2,\dotsc,T_7$, using the same arguments as in $T_1$. Since $N\overset{a}{\approx} d_T^m/(a_T h)^{1+m}$, from \eqref{eqa7}-\eqref{eqa8} we obtain
	\begin{equation}\label{eqa9}
		\sum_{k=1}^7 T_k=O\Big(d_T^m \tfrac{T^{-Md_T^\lambda/(254 L_1+6\bar K)}}{(a_Th)^{1+m}}\Big)+O\Big( d_T^m\tfrac{Th(a_T\tau_T)^{1+\beta}}{(a_Th)^{1+m}}\Big)\coloneqq O(S_1)+O(S_2).
	\end{equation}
	If $\lambda>0$, then $d_T^\lambda\to\infty$ and $S_1=o(1)$ for any $M>0$. If $\lambda=0$, choose $M$ sufficiently large to satisfy $S_1=o(1)$. In both cases, $S_1=o(1)$. For $S_2$, after straightforward algebraic manipulations, we use condition $1/h=O(T^{\theta}/\ln(T))$ in  \eqref{eq9_} to obtain that  
	\begin{align*}
		S_2&= T^{mr+1-\tfrac{(s-2)(1+\beta)}{2(s-1)}+\tfrac{1+m}{2}} 
        h^{\tfrac{1-m}{2}-\tfrac{s(1+\beta)}{2(s-1)}}
        (\ln T)^{\tfrac{(s-2)(1+\beta)}{2(s-1)}-\tfrac{1+m}{2}}\\
        &=O\Big( T^{\tfrac{1}{2(s-1)}\big\{\beta(2-s)+m(s-1)(1+2c)+2s-1+\theta\big[\beta s+m(s-1)+1\big]\big\}}\underbrace{(\ln T)^{-m-\tfrac{1+\beta}{s-1}}}_{=o(1)}\Big)
	\end{align*}
	since the exponent of $h$ is a strictly negative number and $mr\leq mc$.  Thus, using condition \eqref{eq8_},  we obtain $S_2=o(1)$.
    Therefore, $$\sup_{(x,\gamma)\in A}\abs{R_{2,T}(x,\gamma)-ER_{2,T}(x,\gamma)}=O_p(d_T^\lambda a_T).$$ Combining this result with \eqref{eqcentral}, we obtain $\sup_{(x,\gamma)\in A}|\hat\Psi(x,\gamma)-E\hat\Psi(x,\gamma)|=O_p(d_T^\lambda a_T)$, as desired.

It remains to show inequalities \eqref{eqclaim1} and \eqref{eqclaim2}. Fix $l\in\{1,\dotsc,N\}$ and $(x,\gamma)\in A_l$. Note that 
\begin{align*}
|\pi_{i,j,T}(x)|&=\bigg|\frac{i/T-x}{h}\bigg|^j=\bigg|\frac{i/T-x_l}{h}+\frac{x_l-x}{h}\bigg|^j\leq 2^{j-1}\bigg(\bigg|\frac{i/T-x_l}{h}\bigg|^j+\bigg|\frac{x_l-x}{h}\bigg|^j\bigg)\\
&\leq 2^{j-1}\big(|\pi_{i,j,T}(x_l)|+(a_Th)^j\big)\leq 2^{j-1}\big(|\pi_{i,j,T}(x_l)|+1\big)\\
&=2^{j-1}\big(|\pi_{i,j,T}(x_l)|+|\pi_{i,0,T}(x_l)|\big).
\end{align*}
Hence, using Lemma \ref{l3} with $\delta=a_T$, we have
    \begin{equation}
        \abs{\epsilon^*_{i,T}(\gamma)}|\pi_{i,j,T}(x)| \abs{k_{i,T}(x)-k_{i,T}(x_l)}\leq 2^{j-1}\big(|\pi_{i,j,T}(x_l)|+|\pi_{i,0,T}(x_l)|\big)a_T\abs{\epsilon^*_{i,T}(\gamma)k^*_{i,T}(x_l)},
    \end{equation}
    which immediately implies \eqref{eqclaim1}. 
    As for \eqref{eqclaim3}, note that the difference of powers factorization and the binomial theorem yield 
    \begin{equation*}
        |\pi_{i,j,T}(x)-\pi_{i,j,T}(x_l)|\leq j2^j\bigg|\frac{x-x_l}{h}\bigg|\big(|\pi_{i,d,T}(x_l)|+|\pi_{i,0,T}(x_l)|\big)\leq j2^ja_T \big(|\pi_{i,d,T}(x_l)|+|\pi_{i,0,T}(x_l)|\big),
    \end{equation*}
    where $d=\max\{j-1,0\}$. Therefore, since $|K(u)|\leq |K^*(u)|$ for $u\in\R$,
    \begin{equation*}
        |\epsilon^*_{i,T}(\gamma)||k_{i,T}(x_l)||\pi_{i,j,T}(x)-\pi_{i,j,T}(x_l)|\leq  j2^ja_T |\epsilon^*_{i,T}(\gamma)||k_{i,T}^*(x_l)| \big(|\pi_{i,d,T}(x_l)|+|\pi_{i,0,T}(x_l)|\big)
    \end{equation*}
    leading to the desired inequality. Next, we show inequality \eqref{eqclaim2}. 
    From the decomposition
    \begin{align*}
    \epsilon_{i,T}^*(\gamma)-\epsilon_{i,T}^*(\gamma_l)
    &=\big(\epsilon_{i,T}(\gamma)-\epsilon_{i,T}(\gamma_l)\big)
    I\big(|\epsilon_{i,T}(\gamma)|\leq \tau_T, |\epsilon_{i,T}(\gamma_l)|\leq \tau_T\big)\\
    &\qquad +\epsilon_{i,T}(\gamma) I\big(|\epsilon_{i,T}(\gamma)|\leq \tau_T, |\epsilon_{i,T}(\gamma_l)|> \tau_T\big)\\
    &\qquad -\epsilon_{i,T}(\gamma_l)
I\big(|\epsilon_{i,T}(\gamma_l)|\leq \tau_T, |\epsilon_{i,T}(\gamma)|> \tau_T\big),
\end{align*}
we apply the triangle inequality to obtain
    \begin{align}\notag
        |\epsilon^*_{i,T}(\gamma)-\epsilon^*_{i,T}(\gamma_l)|&\leq |\epsilon_{i,T}(\gamma)-\epsilon_{i,T}(\gamma_l)|
        +|\epsilon_{i,T}(\gamma)| I\big(|\epsilon_{i,T}(\gamma_l)|>\tau_T\big)\\ \notag
        &\qquad +|\epsilon_{i,T}(\gamma_l)| I\big(|\epsilon_{i,T}(\gamma)|>\tau_T\big)\\\label{eqdanremen1}
        &\leq |\epsilon_{i,T}(\gamma)-\epsilon_{i,T}(\gamma_l)|+2Y_{i,T}I(Y_{i,T}>\tau_T),
    \end{align} 
where $Y_{i,T}=\sup_\gamma|\epsilon_{i,T}(\gamma)|$. By the Lipschitz property in Assumption A.3\eqref{eqlipz}, we have that almost surely (a.s.)
    \begin{equation}\label{eqc1}
        \abs{k_{i,T}(x_l)} |\epsilon_{i,T}(\gamma)-\epsilon_{i,T}(\gamma_l)|\leq \abs{k_{i,T}^*(x_l)} \xi_{i,T}(\gamma_l)\|\gamma-\gamma_l\|\leq a_T h \abs{k_{i,T}^*(x_l)}\xi_{i,T}(\gamma_l),
    \end{equation}
     since $\|\gamma-\gamma_l\|\leq a_T h$ and $|K(u)|\leq |K^*(u)|$ for all $u\in\R$. Inequalities \eqref{eqdanremen1}-\eqref{eqc1} yields \eqref{eqclaim2}.  Finally, we show \eqref{eqdanremen3}.  Since $K^*$ is compactly supported on $[-2L_1,2L_1]$ and $|K^*(u)|\leq \Lambda_1$,
     \begin{align*}
       \max_{1\leq l\leq N}U_{l,T}&\leq   \frac{2}{Th}\sum_{i=1}^T \max_{1\leq l\leq N}\big\{\abs{k_{i,T}^*(x_l)}|\pi_{i,k,T}(x_l)| \big\}Y_{i,T}I(Y_{i,T}>\tau_T)\\
       &\leq \frac{2^k L_1^k\Lambda_1}{Th}\sum_{i=1}^T Y_{i,T}I(Y_{i,T}>\tau_T).
    \end{align*}
     
     After applying expectations, inequality \eqref{eqa1} implies that
     \begin{align*}
        \E\bigg(\max_{1\leq l\leq N} U_{l,T}\bigg)&\leq\frac{2^k L_1^k\Lambda_1}{Th}\sum_{i=1}^T \E\big(Y_{i,T}I(Y_{i,T}>\tau_T)\big)\leq \frac{c_2}{h}d_T^\lambda\tau_T^{1-s}=c_2d_T^\lambda a_T ,
     \end{align*}
    for some $c_2>0$. By Markov's inequality, for all $\delta>0$, the choice $M=c_2/\delta$ gives
    \begin{align*}
        P\Big(\max_{1\leq l\leq N} U_{l,T}>Md_T^\lambda a_T\Big)&\leq \frac{1}{M d_T^\lambda a_T}\E\Big(\max_{1\leq l\leq N} U_{l,T}\Big)\leq \delta,
    \end{align*}
    as desired.
	\qed
	\vspace{6mm}

    \noindent\textbf{Technical Remarks.}
The following aspects of the proof of Theorem \ref{teo2} are particularly relevant when compared with the arguments of \cite{kristensen} and \cite{hansen}. 
\begin{enumerate}
    \item[(i)] Since $\sup_{x,\gamma} |\E R_{1,T}(x,\gamma)| 
\leq \E(\sup_{x,\gamma} |R_{1,T}(x,\gamma)|)$, 
pointwise bounds on expectations do not control the tail behavior of 
$\sup_{x,\gamma} |R_{1,T}(x,\gamma)|$ through Markov's inequality alone. 
To guarantee a probability bound for
$P(\sup_{x,\gamma}|R_{1,T}|>C d_T^\lambda a_T)$, we employ a slightly larger truncation level $\tau_T$ and the additional uniform moment bound in \eqref{eqextrab}.
\item[(ii)] Under the fixed grid $x_{i,T}=i/T$, the classical variance order $\sup_{x,\gamma}\sigma^2_{T,m_T}(x,\gamma)=O(m_T h)$ obtained for parameter independent data \cite[see][]{hansen} does not follow from the available deterministic integral approximation (see Lemma \ref{teo1}) unless $m_T$ is of order $T$. However, such a choice is not admissible in the present framework, as it would violate the requirements of Liebscher-Rio's inequality (Lemma \ref{l1}). Consequently, the cardinality $n_{T}(x)\coloneqq \#J_x(L_1)$ plays a central role in the application of this inequality, which highlights the importance of the compactness of $\supp K$.
\item[(iii)] The almost sure local Lipschitz condition (Assumption A.3\eqref{eqlipz})is crucial for controlling $B_{k,T}$ in \eqref{eqclaim2}. It produces  bounds of the form $|\epsilon_{i,T}(\gamma)-\epsilon_{i,T}(\gamma_j)|
\leq \xi_{i,T}(\gamma_j)\|\gamma-\gamma_j\|$ with Lipschitz coefficients evaluated at $\gamma_j$, which is directly compatible with the finite covering argument for $A=[0,1]\times \Theta_T$.
\end{enumerate}

	\textbf{Proof of Theorem \ref{teo3}}
	We use the same notation as in the proof of Theorem \ref{teo2}. Let $$\tau_T\coloneqq \bigg(\frac{T^3\phi_T}{h}\bigg)^{\tfrac{1}{2(s-1)}} \ \text{ and } \ a_T\coloneqq \sqrt{\frac{\ln T}{Th}}.$$ As in \eqref{eqa2}, it follows that $E \big(\sup_{x,\gamma}|R_{1,T}(x)|\big)=O\big(d_T^\lambda \tau_T^{1-s}/h\big)$. Note that 
	\begin{align*}
		\frac{\tau_T^{1-s}}{a_T h}=\frac{1}{T\ln T(\ln\ln T)^2}.
	\end{align*}
	Since $\sum_{T=3}^\infty 1/(T\ln T(\ln\ln T)^2)<\infty$ \citep[see p. 63 of][]{rudin}, Markov's inequality gives, for all $C_0>0$
	\begin{align*}
		\sum_{T=1}^\infty P\Big( \sup_{x,\gamma}\lvert &R_{1,T}(x,\gamma)-ER_{1,T}(x,\gamma)\rvert> C_0 d_T^\lambda a_T\Big) \leq \sum_{T=1}^\infty \frac{2 E \big(\sup_{x,\gamma}|R_{1,T}(x)|\big)}{C_0 d_T^\lambda a_T}\\
        &\leq \frac{2C}{C_0}\bigg\{2+\sum_{T=3}^\infty \frac{1}{T\ln T (\ln \ln T)^2}\bigg\}<\infty.
	\end{align*}
	Then, by Borel-Cantelli's lemma, for all $C_0>0$
	\begin{align*}
		&P\Big(\limsup_{T\to\infty} \Big\{\sup_{x,\gamma}\lvert R_{1,T}(x,\gamma)-ER_{1,T}(x,\gamma)\rvert> C_0 d_T^\lambda a_T\Big\}\Big)=0,
	\end{align*}
	which implies $\sup_{x,\gamma}\lvert R_{1,T}(x)-ER_{1,T}(x)\rvert=o\big(d_T^\lambda a_T\big)$ almost surely.
	
	Next, we verify that \eqref{eqa8} remains valid for $\tau_T=(T^{3}\phi_T/h)^{1/(2(s-1))}$ and	 $A_l=\{(x,\gamma):|x-x_l|\leq a_T h(\ln\ln T)^2, \|\gamma-\gamma_l\|\leq a_T h(\ln  \ln T)^2\}$, so that $N\overset{\text{a}}{\approx} d_T^m/(a_Th(\ln\ln T)^2)^{1+m}$. Note that for each $A_l$, we have  $\|\gamma-\gamma_l\|\leq a_T h(\ln  \ln T)^2\leq h$ for all $T$ large enough, since $a_T(\ln  \ln T)^2=o(1)$. Therefore, Assumption A.3(3) applies on each $A_l$ for $T$ large enough. In particular, 
    \begin{align*}
        a_T\tau_T=T^{\tfrac{3}{2(s-1)}-\tfrac12}h^{-\tfrac12-\tfrac{1}{2(s-1)}}(\ln T)^{\tfrac12}\phi_T^{\tfrac{1}{2(s-1)}}=T^{\tfrac{3}{2(s-1)}-\tfrac12+\theta\big[\tfrac12+\tfrac{1}{2(s-1)}\big]}(\ln\ln T)^{-2}.
    \end{align*}
    But the exponent of $T$ is negative if and only if $\theta<1-4/s$, which is  true for $\theta$ given in  \eqref{eq11_}. Thus, $a_T\tau_T=o(1)$ and the choice $m_T\coloneqq (a_T\tau_T)^{-1}\to\infty$ satisfies the conditions of Lemma \ref{l1}. Choosing $M>0$ sufficiently large so that $mr+(1+\theta)\big[\tfrac{m+1}{2}\big]-\tfrac{M d_T^\lambda}{254 L_1+6\bar K}<-1$, it follows that
	\begin{align*}
		S_1&=\frac{T^{mr+\tfrac{m+1}{2}-\tfrac{M d_T^\lambda}{254 L_1+6\bar K}}}{(h\phi_T)^{(m+1)/2}} =O\Big(T^{mr+(1+\theta)\big[\tfrac{m+1}{2}\big]-\tfrac{M d_T^\lambda}{254 L_1+6\bar K}}\phi_T^{-1-m}\Big)=O\big((T\phi_T)^{-1}\big),
	\end{align*}
	since $h^{-(m+1)/2}=O(T^{\theta(m+1)/2}\phi_T^{-(m+1)/2})$ by condition \eqref{eqcond}. For $S_2$, 
    we  have from \eqref{eq11_},
   \begin{align*}
       S_2&\leq d_T^m\tfrac{T(a_T\tau_T)^{1+\beta}}{(a_Th(\ln\ln T)^2)^{1+m}} \\
       &= T^{1+mr+\tfrac{1+m}{2}+\tfrac{(1+\beta)(4-s)}{2(s-1)}}
       h^{-\tfrac{(1+\beta)}{2(s-1)}-\tfrac{m+1}{2}}
       \phi_T^{\tfrac{(1+\beta)}{2(s-1)}-\tfrac{1+m}{2}}
       (\ln\ln T)^{-2(2+\beta+m)}\\
       &=o\Big(T^{1+mr+\tfrac{1+m}{2}+\tfrac{(1+\beta)(4-s)}{2(s-1)}+\theta\big[\tfrac{(1+\beta)}{2(s-1)}+\tfrac{m+1}{2}\big]}
       \phi_T^{-(1+m)}\Big)
       \\
       &=o((T\phi_T)^{-1}),
   \end{align*}
   since $h^{-\tfrac{(1+\beta)}{2(s-1)}-\tfrac{m+1}{2}}=o\Big(T^{\theta\big[\tfrac{(1+\beta)}{2(s-1)}+\tfrac{m+1}{2}\big]}\phi_T^{-\tfrac{(1+\beta)}{2(s-1)}-\tfrac{m+1}{2}}\Big)$ by condition \eqref{eqcond}.
Since $S_1+S_2=O((T\phi_T)^{-1})$ and $\sum_{T=3}^\infty (T\phi_T)^{-1}<\infty$, using the same decomposition as in \eqref{eqa7}, we obtain for all $C_0>0$
\begin{equation*}
    \sum_{T=1}^\infty P\Big( \sup_{x,\gamma}\lvert R_{2,T}(x,\gamma)-ER_{2,T}(x,\gamma)\rvert> C_0 d_T^\lambda a_T\Big) \leq C \sum_{T=1}^\infty (S_1+S_2)<\infty.
\end{equation*}
Hence, the Borel-Cantelli's lemma yields $\sup_{x,\gamma}\lvert R_{2,T}(x,\gamma)-ER_{2,T}(x,\gamma)\rvert=o(d_T^\lambda a_T)$ almost surely. Combining the bounds for $R_{1,T}$ and $R_{2,T}$ gives the desired result.\qed

	
	\textbf{Proof of Theorem \ref{teo4}}
    We first focus on estimator $\hat g$. 
	Write
	\begin{equation}\label{eqinit}
		\abs{\hat g(x)-g(x)}\leq \abs{\hat g(x)-E\hat g(x)}+\abs{E\hat g(x)-g(x)}\coloneqq A_{1,T}+A_{2,T}, \quad \forall x\in[0,1].
	\end{equation}
	 From Assumptions A.2-A.4 and Lemma \ref{l6}, the conditions of Lemma 1.3 and Proposition 1.12 of \cite{tsybakov} are satisfied. In particular, for the local linear weights 
$W_{t,T}(\cdot)$ defined in \eqref{eqpesos} we have, uniformly in $x\in[0,1]$,
\begin{equation*}
    \sum_{t=1}^T W_{t,T}(x)=1, \quad 
    \sum_{t=1}^T W_{t,T}(x)\,(t/T-x)=0,\quad \sup_{1\leq t\leq T}|W_{t,T}(x)|\leq \frac{C}{Th}, \quad \sum_{t=1}^T |W_{t,T}(x)|\leq C,
\end{equation*}
with only $k_T=O(Th)$ nonzero terms due to the compact support of $K$ (see Lemma \ref{l2}). 
Moreover, since $g\in C^2[0,1]$, its derivative $g'$ is Lipschitz continuous.
Hence, using Tsybakov's results and the Taylor expansion with Lagrange reminder, we have
	\begin{align}\notag
		A_{2,T}&=\abs[\bigg]{\sum_{t=1}^T W_{t,T}(x)\chave[\big]{g(t/T)-g(x)} }\\\notag
		&=\abs[\bigg]{\sum_{t=1}^T W_{t,T}(x)\chave[\big]{g(x)+g'\colc{x+\tau_t(t/T-x)}(t/T-x)-g(x)} }\\\notag
		&=\abs[\bigg]{\sum_{t=1}^T W_{t,T}(x)\chave[\big]{g'\colc{x+\tau_t(t/T-x)}(t/T-x)} -\sum_{t=1}^T W_{t,T}(x)(t/T-x)g'(x)}\\\notag
		&\leq \sum_{t=1}^T \abs{W_{t,T}(x)} \abs{t/T-x }\abs[\big]{g'(x+\tau_t(t/T-x))-g'(x)}\\\notag
		&\leq C \sum_{t=1}^T \abs{W_{t,T}(x)} \abs{t/T-x }^2= C \sum_{t=1}^T \abs{W_{t,T}(x)} \abs{t/T-x }^2I\parent[\Big]{\abs[\Big]{\tfrac{t/T-x}{h}}\leq 1}\\\label{eqbias}
		&\leq C\sum_{t=1}^T \sup_x \abs{W_{t,T}(x)} h^2\leq C h^2,
	\end{align}
	uniformly in $x\in[0,1]$, for all $T$ sufficiently large, where $\tau_t\in(0,1)$. Thus, $\sup_{x\in[0,1]}A_{2,T}=O(h^2)$. 
	
	We now turn to the term $A_{1,T}$. Write
	\begin{equation*}
		A_{1,T}=\abs{e_1^\intercal S_{T,x}^{-1} D_{T,x}^V}\qquad ,
	\end{equation*}
	where  $S_{T,x}$ is defined as in \eqref{eq13_} and
    \begin{align*}
        D_{T,x}^V&=T^{-1}\left[\begin{array}{c} \sum_{t=1}^T V_{t,T}K_h(t/T-x)\\
		\sum_{t=1}^T V_{t,T}K_h(t/T-x)((t/T-x)/h)\end{array} \right]\coloneqq \left[\begin{array}{c} d_{T,0}^V(x)\\
		d_{T,1}^V(x)\end{array} \right].
    \end{align*}
	
	Using Cauchy-Schwarz inequality and Lemma \ref{l7} with $v=\big(d_{T,0}^V(x),d^V_{T,1}(x)\big)^\intercal$ yield 
\begin{equation*}
	\sup_{x\in[0,1]}\abs[\Bigg]{e_{1}^\intercal S_{T,x}^{-1}\left[
		\begin{array}{c}
			 d^V_{T,0}(x)  \\ 
			 d^V_{T,1}(x) 
		\end{array}
		\right]}\leq \sup_{x\in[0,1]}\underbrace{\norm{e_{1}}}_{=1} \norm[\bigg]{S_{T,x}^{-1}\left[
		\begin{array}{c}
			d^V_{T,0}(x)  \\ 
			d^V_{T,1}(x) 
		\end{array}
		\right]}<\frac{1}{\lambda_0}\sup_{x\in[0,1]}\norm[\bigg]{\left[
		\begin{array}{c}
			d^V_{T,0}(x) \\ 
			d^V_{T,1}(x)
		\end{array}
		\right]},
\end{equation*}
for all $T$ large enough. Therefore, by the equivalence of the $L_2$ and $L_1$ norms in finite-dimensional real spaces \citep[Corollary 5.4.5 of][]{horn} and Theorem \ref{teo2}, 
\begin{align*}
	\sup_{x\in[0,1]}A_{1,T}\leq\frac{1}{\lambda_0}\sup_{x\in[0,1]}\norm[\bigg]{\left[
		\begin{array}{c}
			d^V_{T,0}(x) \\ 
			d^V_{T,1}(x)
		\end{array}
		\right]}\leq \frac{1}{\lambda_0}\left(\sup_{x\in[0,1]}\abs[\big]{d^V_{T,0}(x)}+\sup_{x\in[0,1]}\abs[\big]{d^V_{T,1}(x)}\right)=O_p\parent{a(h)}.
\end{align*}
By Assumption A.6 and Proposition 1 of \cite{orbe}, the array $\{V_{t,T}\}$ is geometrically strongly mixing and possesses uniformly bounded moments of order $s>2$. Hence, the conditions of Theorem \ref{teo2} are satisfied for the zero-mean array $\{V_{t,T}\}$, and we obtain
\begin{equation*}
    \sup_{x\in[0,1]}\abs[\big]{d^V_{T,0}(x)}+\sup_{x\in[0,1]}\abs[\big]{d^V_{T,1}(x)}=O_p\parent{a(h)},
\end{equation*}
 where $a(u)\coloneqq \sqrt{\ln(T)/(Tu)}$.  Consequently, $\sup_{x\in[0,1]}A_{1,T}=O_p(a(h))$.
Combining this result with \eqref{eqbias} and \eqref{eqinit}, we have 
\begin{equation}\label{eqglin}
    \sup_{x\in[0,1]}\abs{\hat g(x)-g(x)}=O(h^2)+O_p(a(h)).
\end{equation}

Next, we show that $\sup_{x\in\I_T}\abs{\hat\phi(x)-\phi(x)}=O(v^2)+O_p(a(v))$. Define the unfeasible estimator
 \begin{equation}\label{equnf2}
        \tilde \phi(x)=\frac{\tilde \Psi_1(x)}{\tilde \Psi_2(x)}\coloneqq  \frac{1/T\sum_{t=2}^T G_v(x_t-x) V_{t,T}   V_{t-1,T}}{1/T\sum_{t=2}^T G_v(x_t-x)  V_{t-1,T}^2},\quad x\in\I_T.
    \end{equation}
where $x_t=t/T$, and write $$\hat V_{t,T}=Y_{t,T}-\hat g(t/T)=\big(g(t/T)-\hat g(t/T)\big)+V_{t,T}\coloneqq r_{t,T}+V_{t,T},\quad 1\leq t\leq T.$$ A direct expansion gives for $x\in\I_T$,
\begin{align*}
    \hat\phi(x)&=\frac{1/T \sum_{t=2}^T G_v(x_t-x) \Big[V_{t,T}V_{t-1,T}
+V_{t,T}r_{t-1,T}
+V_{t-1,T}r_{t,T}
+r_{t,T}r_{t-1,T}\Big]}{1/T \sum_{t=2}^T G_v(x_t-x) \Big[V_{t-1,T}^2
+2V_{t-1,T}r_{t-1,T}
+r_{t-1,T}^2\Big]}\\
&=\frac{\tilde \Psi_1(x)+1/T \sum_{t=2}^T G_v(x_t-x) \Big[V_{t,T}r_{t-1,T}
+V_{t-1,T}r_{t,T}
+r_{t,T}r_{t-1,T}\Big]}{\tilde \Psi_2(x)+1/T \sum_{t=2}^T G_v(x_t-x) \Big[
2V_{t-1,T}r_{t-1,T}
+r_{t-1,T}^2\Big]}\\
&\coloneqq \frac{\tilde \Psi_1(x)+R_{1,T}(x)}{\tilde \Psi_2(x)+R_{2,T}(x)}.
\end{align*}
Hence
\begin{align}\notag
    |\hat\phi(x)-\tilde\phi(x)|&=\bigg|\frac{\tilde \Psi_1(x)+R_{1,T}(x)}{\tilde \Psi_2(x)+R_{2,T}(x)}-\frac{\tilde \Psi_1(x)}{\tilde \Psi_2(x)}\bigg|\\\label{eqdes1}
    &\leq \frac{|R_{1,T}(x)|}{|\tilde \Psi_2(x)+R_{2,T}(x)|}+|\tilde\phi(x)|\frac{|R_{2,T}(x)|}{|\tilde \Psi_2(x)+R_{2,T}(x)|}.
\end{align}
By applying Theorem \ref{teo2} to the zero mean data $ V_{t,T}$ with kernel $G$ and bandwidth $v$, we have 
\begin{equation*}
   \sup_{x\in\I_T}\bigg| \frac{1}{T}\sum_{t=2}^T G_v(x_t-x)V_{t,T}\bigg|=O_p(a(v)) \quad \text{ and }\quad
    \sup_{x\in\I_T}\bigg|\frac{1}{T}\sum_{t=2}^T G_v(x_t-x)V_{t-1,T}\bigg|=O_p(a(v)).
\end{equation*}
Moreover, Lemma \ref{l2} yields $ \sup_{x\in\I_T}|\frac{1}{T}\sum_{t=2}^T G_v(x_t-x)|=O(1)$.
Therefore, using $\sup_t |r_{t,T}|=O_p(h^2+a(h))$,
\begin{equation}\label{eqR1}
    \sup_{x\in\I_T}|R_{i,T}(x)|=O_p((h^2+a(h))a(v)+(h^2+a(h))^2)=o_p(h^2+a(h)),\quad i\in\{1,2\},
\end{equation}
 since $h^2+a(h)=o(1)$ and $v\overset{a}{\approx}c_v h$ implies $a(v)=O(a(h))=o(1)$. On the other hand, we argue that the denominator $\inf_{x\in\I_T}|\tilde \Psi_2(x)+R_{2,T}(x)|$ is bounded away from zero with probability approaching to one. As shown below in \eqref{eqconclu2}, uniformly in $x\in\I_T$ and for all $T$ sufficiently large
 \begin{equation*}
     \E(\tilde \Psi_2(x))=\Lambda(x)+O(v^2+1/T),\quad \Lambda(x)\coloneqq \sigma^2_e/(1-[\phi(x)]^2), 
 \end{equation*}
 and by A.6(i), we have  $\inf_{x\in\I_T} \Lambda(x)\geq \sigma^2_e/(1-\bar \phi^2)\coloneqq c_e>0$. In addition, Theorem \ref{teo2} gives $\sup_{x\in\I_T} |\tilde\Psi_2(x)-\E \tilde\Psi_2(x)|=O_p(a(v))=o_p(1)$. Thus, by the triangle inequality 
\begin{equation*}
    \inf_{x\in\I_T}  \big|\tilde\Psi_2(x)\big|\geq \inf_{x\in\I_T} \big|\E \tilde\Psi_2(x) \big|-\sup_{x\in\I_T} |\tilde\Psi_2(x)-\E \tilde\Psi_2(x)|\geq c_e-\sup_{x\in\I_T} |\tilde\Psi_2(x)-\E \tilde\Psi_2(x)|.
\end{equation*}
Therefore, if $ \sup_{x\in\I_T} |\tilde\Psi_2(x)-\E \tilde\Psi_2(x)|\leq c_e/2$ , then $ \inf_{x\in\I_T}  \big|\tilde\Psi_2(x)\big|\geq c_e/2$. Thus, using the monotonicity of the probability measure,
\begin{align}\notag
    P\Big(\inf_{x\in\I_T}  \big|\tilde\Psi_2(x)\big|\geq c_e/2\Big)&\geq P\Big(\sup_{x\in\I_T} |\tilde\Psi_2(x)-\E \tilde\Psi_2(x)|\leq c_e/2\Big)\\\label{eqsustain}
    &=1-P\Big(\sup_{x\in\I_T} |\tilde\Psi_2(x)-\E \tilde\Psi_2(x)|>c_e/2\Big)\to 1,\quad T\to\infty.
\end{align}
Applying again the triangle inequality, if $\sup_{x\in\I_T} |R_{2,T}|\leq c_e/4$ and $\inf_{x\in\I_T}  \big|\tilde\Psi_2(x)\big|\geq c_e/2$, then 
\begin{align*}
    \inf_{x\in\I_T} |\tilde \Psi_2(x)+R_{2,T}(x)|\geq \inf_{x\in\I_T} |\tilde \Psi_2(x)|-\sup_{x\in\I_T} |R_{2,T}(x)|\geq c_e/4,
\end{align*}
so, by the monotonicity and subadditivity of $P$, and using \eqref{eqR1}-\eqref{eqsustain},
\begin{align*}
    P\Big(\inf_{x\in\I_T} |\tilde \Psi_2(x)+&R_{2,T}(x)|\geq c_e/4\Big)\geq P\Big(\sup_{x\in\I_T} |R_{2,T}|\leq c_e/4 \ \text{ and }  \inf_{x\in\I_T} \big|\tilde\Psi_2(x)\big|\geq  c_e/2\Big)\\
    &=1- P\Big(\sup_{x\in\I_T} |R_{2,T}|> c_e/4 \ \text{ or }  \inf_{x\in\I_T} \big|\tilde\Psi_2(x)\big|<  c_e/2\Big)\\
    &\geq 1-P\Big(\sup_{x\in\I_T} |R_{2,T}|> c_e/4\Big)-P\Big(\inf_{x\in\I_T}  \big|\tilde\Psi_2(x)\big|< c_e/2\Big)\to 1,\quad T\to\infty.
\end{align*}
Hence, $\sup_{x\in\I_T} \{1/|\tilde \Psi_2(x)+R_{2,T}(x)|\}=O_p(1)$. Assuming \eqref{equnf0} holds, which is proved below in \eqref{eqfinali}, 
\begin{equation}\label{equnf0}
    \sup_{x\in\I_T }|\tilde \phi (x)-\phi(x)|=O_p(v^2+a(v))=o_p(1),
\end{equation}
it follows that
\begin{align*}
    \sup_{x\in\I_T}|\tilde\phi(x)|\leq \sup_{x\in\I_T }|\tilde \phi (x)-\phi(x)|+\sup_{x\in\I_T }|\phi(x)|=o_p(1)+O(1)=O_p(1),
\end{align*}
 since  $\phi$ is bounded on $[0,1]$. Taking suprema in \eqref{eqdes1} and using \eqref{eqR1}-\eqref{equnf0}, we obtain
 \begin{align*}
    \sup_{x\in\I_T}|\hat\phi(x)-\tilde\phi(x)|&\leq\sup_{x\in\I_T}\bigg\{\frac{1}{|\tilde \Psi_2(x)+R_{2,T}(x)|}\bigg\}\Big(\sup_{x\in\I_T}|R_{1,T}(x)|+\sup_{x\in\I_T}|\tilde\phi(x)|\sup_{x\in\I_T}|R_{2,T}(x)|\Big)\\
    &=o_p(h^2+a(h)).
\end{align*}
Thus, provided \eqref{equnf0} holds, we have that
\begin{equation}\label{equnf}
     \hat \phi(x)=\tilde\phi(x)+o_p(h^2+a(h)),
\end{equation}
uniformly in $x\in\I_T$.  

By Assumption A.6, the conditions of Proposition 1 in \cite{orbe} are satisfied which, in turn, is used to guarantee that $\{V_{t,T}\}$ attends Assumption A.1, for all $\beta>0$ and some $s>2$. We then use Theorem \ref{teo2}
to obtain $\sup_{x\in[0,1]}\big|\tilde \phi(x)-\E(\tilde \phi(x))\big|=O_p(a(v))$. By \cite{dahlhaus}, $\E(V_{t,T}^2)=\Lambda(t/T)+o(1/T)$ for all $1\leq t\leq T$, where $\Lambda(t/T)\coloneqq \sigma^2_e/(1-[\phi(t/T)]^2)$ and $\sigma^2_e\coloneqq\var(e_{t,T})$. From Assumptions A.6 and A.5, $\Lambda(\cdot)$ is twice continuously differentiable on $[0,1]$ and $\sup_{u\in[0,1]}|\Lambda'(u)|\leq C$. Thus, the Mean Value Theorem implies $\E(V_{t-1,T}^2)=\Lambda((t-1)/T)=\Lambda(t/T)+O(1/T)$ uniformly in $t\in\{2,\cdots, T\}$.
Therefore, by integral approximations by finite sums (along the lines of the proof of Lemma \ref{l2}), second-order Taylor expansion and Weierstrass Extreme Value theorem, we obtain 
\begin{align}\notag
    \E\left(\tilde \Psi_1(x)\right)&=\frac{1}{T}\sum_{t=2}^T G_v(t/T-x)\phi(t/T)\E(V_{t-1,T}^2+V_{t-1,T}e_{t,T})\\\notag
    &=\frac{1}{T}\sum_{t=2}^T G_v(t/T-x)\phi(t/T)\left[\Lambda(t/T)+O(1/T)\right]\\\notag
    &=\frac{1}{v}\int_0^1 G\left(\tfrac{u-x}{v}\right)\phi(u)\Lambda(u)du+O(1)\frac{1}{Tv}\int_0^1 G\left(\tfrac{u-x}{v}\right)du+O(1/T)\\\notag
    &=\int_{-x/v}^{(1-x)/v} G\left(w\right)\phi(x+wv)\Lambda(x+wv)dw+O(1/T)\\\label{eqconclu1}
    &=\phi(x)\Lambda(x)+O(v^2)+O(1/T)\coloneqq \Psi_1(x)+O(v^2+1/T),
\end{align}
uniformly in $x\in \I_T$, where we use the fact that $G(\cdot)$ is a second order kernel satisfying $\int_{-1}^1G(w)dw=1$, and that, for $T$ large enough, $[-1,1]\subseteq[-x/v, (1-x)/v]$ uniformly in $x\in\I_T$. In particular,  $x+wv\in(0,1)$ for all $|w|\leq 1$. Analogously, it holds that uniformly in $x\in \I_T$,
\begin{equation}\label{eqconclu2}
    \E\left(\tilde \Psi_2(x)\right)=\Lambda(x)+O(v^2)+O(1/T)\coloneqq \Psi_2(x)+O(v^2+1/T).
\end{equation}
Next, using the triangle inequality and the Mean Value theorem, we have that
\begin{align}\notag
    \big|\tilde\phi(x)-\phi(x)\big|&= \bigg|\frac{\tilde\Psi_1(x)}{\tilde\Psi_2(x)}-\frac{\Psi_1(x)}{\Psi_2(x)}\bigg|=\bigg|\frac{\tilde\Psi_1(x)-\Psi_1(x)}{\tilde\Psi_2(x)}+\frac{\Psi_1(x)\big(\Psi_2(x)-\tilde \Psi_2(x)\big)}{\Psi_2(x)\tilde \Psi_2(x)}\bigg|\\\label{eqconclu3}
    &\leq \frac{1}{\Psi_2^*(x)}\big|\tilde\Psi_1(x)-\Psi_1(x)\big|+\frac{|\Psi_1^*(x)|^{\ }}{[\Psi_2^*(x)]^2}\big|\Psi_2(x)-\tilde \Psi_2(x)\big|,
\end{align}
where $\Psi_j^*(x)=\Psi_j(x)+\tau_j(\tilde\Psi_j(x)-\Psi_j(x))$ for some $\tau_j\in[0,1]$, $j\in\{1,2\}$. As argued by \cite{kristensen}, $1/\Psi_2^*(x)$ and $|\Psi_1^*(x)|/[\Psi_2^*(x)]^2$ are $O_{a.s.}(1)$ uniformly in $x\in\I_T$. Hence, from \eqref{eqconclu1}-\eqref{eqconclu3}, and using Theorem \ref{teo2} once more, 
\begin{equation}\label{eqfinali}
    \sup_{x\in\I_T}\big|\tilde\phi(x)-\phi(x)\big|=O_{a.s.}(1)\Big(O_p(a(v))+O(v^2+1/T)\Big)=O_p\big(a(v)+v^2\big),
\end{equation}
since $1/T=o(a(v))$. This proves \eqref{equnf0}, and so \eqref{equnf} is valid. Combining this result with \eqref{equnf} yields
\begin{equation*}
     \sup_{x\in\I_T}\big|\hat\phi(x)-\phi(x)\big|=O_p\big(a(h)+h^2\big)+O_p\big(a(v)+v^2\big)=O_p\big(a(h)+h^2\big).
\end{equation*}
\qed

\textbf{Proof of Corollary \ref{corol1}} Following the notations in the proof of Theorem \ref{teo4}, the bias term satisfies
\begin{equation*}
    A_{2,T}=\abs[\bigg]{\sum_{t=1}^T W_{t,T}(x)\chave[\big]{g(t/T)-g(x)} }=O(h^2)
\end{equation*}
uniformly in $x\in[0,1]$ and $\phi\in[-\bar\phi,\bar\phi]$ because the weight $W_{t,T}(x)$ is independent of $\phi$. The stochastic term $A_{1,T}$ is controlled by applying Theorem \ref{teo2} to terms of form
\begin{equation*}
    \frac{1}{T}\sum_{i=1}^T V_{t,T}(\phi)K_h(t/T-x)((t/T-x)/h)^j,\qquad j\in\{0,1\}.
\end{equation*}
Thus, to obtain the uniformity over $|\phi|\leq \bar\phi$, it is sufficient to show that $\{V_{t,T}(\phi)\}$ satisfies the parameter dependence conditions of Theorem \ref{teo2}:  for all $T\geq 1$, $t\in[T]$, there exists  $\xi_{t,T}(\phi)\geq 0$, such that almost surely
        \begin{equation}\label{eqcond1}
            |V_{t,T}(\phi')-V_{t,T}(\phi)|\leq \xi_{t,T}(\phi)\|\phi'-\phi\|,\qquad \phi',\phi\in\Theta: \|\phi'-\phi\|\leq h
        \end{equation}
    and
    \begin{align}\label{eqcond2}
    \sup_{T\geq 1}\sup_{1\leq i\leq T} \E\Big(\sup_{|\phi|\leq\bar\phi} |V_{i,T}(\phi)|^s\Big)&\leq C,\\\label{eqcond3}
             \sup_{T\geq 1}\,\sup_{1\leq i\leq T}\E(\sup_{|\phi|\leq \bar\phi}| \xi_{i,T}(\phi)|^s)&\leq C.
	\end{align}
	Under $|\phi|<1$, model \eqref{eqap2} admits the causal moving average representation given by $V_{t,T}(\phi)=\phi^t V_{0,T}(\phi)+\sum_{j=0}^{t-1} \phi^j e_{t-j,T}$, 
    and so $|V_{t,T}(\phi)|^s\leq  2^{s-1} (|V_{0,T}(\phi)|^s+(\sum_{j=0}^{t-1} \bar \phi^j |e_{t-j,T}|)^s$. Note that, by H\"older's inequality, we have
     \begin{align*}
        \sum_{j=0}^{t-1} \bar \phi^j |e_{t-j,T}|= \sum_{j=0}^{t-1} (\bar \phi^{\tfrac{j}{s}} |e_{t-j,T}|)\bar \phi^{\tfrac{j(s-1)}{s}}\leq \bigg(\sum_{j=0}^{t-1} \bar \phi^{j} |e_{t-j,T}|^s\bigg)^{\tfrac{1}{s}} \bigg( \sum_{j=0}^{t-1} \bar \phi^{j}\bigg)^{\tfrac{s-1}{s}}.
    \end{align*}
    Thus
    \begin{align}\notag
        \sup_{t,T}\E\Big(\sup_{|\phi|\leq\bar\phi}|V_{t,T}(\phi)|^s\Big) &\leq 2^{s-1}\bigg[\sup_T\E\Big(\sup_{|\phi|\leq\bar\phi}|V_{0,T}(\phi)|^s\Big)+\sup_{t,T}\E(|e_{t,T}|^s)\sum_{j=0}^\infty \bar \phi^j\bigg]\\\label{eqauxiliar}
        &\leq C\bigg[1+\frac{1}{1-\bar\phi}\bigg]\leq C,
    \end{align}
so \eqref{eqcond2} is verified. Now, we prove \eqref{eqcond1} by induction. By assumption on the initial condition ($t=0$), there exists $\xi_{0,T}$ such that, almost surely (a.e.) for all $\phi,\phi'\in\Theta$ with $|\phi'-\phi|\leq h$, $|V_{0,T}(\phi')-V_{0,T}(\phi)|\leq \xi_{0,T}(\phi)|\phi'-\phi|$. Fix $t\geq 1$. Suppose that there exists a nonnegative function $\xi_{t-1,T}(\phi)$ such that, a.e. for all $\phi,\phi'\in\Theta$ with $|\phi'-\phi|\leq h$, $|V_{t-1,T}(\phi')-V_{t-1,T}(\phi)|\leq \xi_{t-1,T}(\phi)|\phi'-\phi|$. From this induction hypothesis, it follows that  a.e. for all $\phi,\phi'\in\Theta$ with $|\phi'-\phi|\leq h$,
\begin{align}\notag
    |V_{t,T}(\phi')-V_{t,T}(\phi)|&=|\phi' V_{t-1,T}(\phi')-\phi V_{t-1,T}(\phi)|\\\notag
    &=|(\phi'-\phi) V_{t-1,T}(\phi')+\phi( V_{t-1,T}(\phi')-V_{t-1,T}(\phi))|\\\notag
    &\leq |\phi'-\phi| |V_{t-1,T}(\phi')|+|\phi| | V_{t-1,T}(\phi')-V_{t-1,T}(\phi)|\\\label{eqprof1}
    &\leq h |V_{t-1,T}(\phi')|+|\phi|\xi_{t-1,T}(\phi)|\phi'-\phi|,
\end{align}
and
\begin{align}\label{eqprof2}
    |V_{t-1,T}(\phi')|&\leq |V_{t-1,T}(\phi)|+|V_{t-1,T}(\phi')-V_{t-1,T}(\phi)|\leq |V_{t-1,T}(\phi)|+\xi_{t-1,T}(\phi)|\phi'-\phi|.
\end{align}
Plugging \eqref{eqprof2} into \eqref{eqprof1} yields
\begin{align}\notag
    |V_{t,T}(\phi')-V_{t,T}(\phi)|&\leq h  |V_{t-1,T}(\phi)|+\xi_{t-1,T}(\phi)|\phi'-\phi|+|\phi|\xi_{t-1,T}(\phi)|\phi'-\phi|\\\notag
    &=\big(|V_{t-1,T}(\phi)|+(|\phi|+h)\xi_{t-1,T}(\phi)\big)|\phi'-\phi|\\\label{eqprof3}
    &\coloneqq \xi_{t,T}(\phi)|\phi'-\phi|,
\end{align}
as desired. Next, we use the last equality in \eqref{eqprof3} to show that \eqref{eqcond3} holds. Since $|\phi|\leq \bar\phi<1$, choose $\underline c\in(\bar\phi,1)$. Then, for all $|\phi|\leq\bar\phi$ and all $T$ sufficiently large, it holds $|\phi|+h \leq \underline c <1$.
From \eqref{eqauxiliar}, there exists $M>0$, such that for all $T$ sufficiently large,
\begin{align}\notag
    \E\bigg(\sup_{|\phi|\leq \bar \phi}|\xi_{t,T}(\phi)|^s\bigg)&=\E\bigg(\sup_{|\phi|\leq \bar \phi}\big||V_{t-1,T}(\phi)|+(|\phi|+h)\xi_{t-1,T}(\phi)\big|^s\bigg)\\\notag
    &\leq 2^{s-1}\bigg[\E\bigg(\sup_{|\phi|\leq \bar \phi}|V_{t-1,T}(\phi)|^s\bigg)+\sup_{|\phi|\leq\bar\phi}(|\phi|+h)^s\E\bigg(\sup_{|\phi|\leq \bar \phi} |\xi_{t-1,T}(\phi)|^s\bigg)\bigg]\\\label{eqiter}
    &\leq 2^{s-1}\bigg[M+\underline c^s\E\bigg(\sup_{|\phi|\leq \bar \phi} |\xi_{t-1,T}(\phi)|^s\bigg)\bigg],\qquad \forall t\in[T].
\end{align}
By iterating inequality \eqref{eqiter} and using the condition $\sup_{T\geq 1}\E(\sup_{|\phi|\leq \bar\phi}|\xi_{0,T}(\phi)|^s)\leq C$, we obtain 
\begin{align}\notag
    \E(|\sup_{\phi}\xi_{t,T}(\phi)|^s)&\leq  2^{s-1}\big[M\sum_{k=0}^{t-1}\underline c^{ks}+\underline c^{ts}\E(\sup_\phi |\xi_{0,T}(\phi)|^s)\big]\\
    &\leq C\sum_{k=0}^{\infty}\underline c^{k}=C\frac{1}{1-\underline c}\leq C.
\end{align}
\qed

		\section*{Appendix B: Auxiliary Results}\label{apA}
			
	This appendix collects several auxiliary lemmas (from \ref{l2.1} to \ref{teo1}) used in the proofs presented in Appendix A. For brevity, the complete proofs are deferred to the supplementary material.

The quantity $\hat \Psi (x,\gamma)$ involves a sum over the set of indices $\{i\}_{i=1}^T$. If the kernel function is  supported on $[-L_1, L_1]$, we only need to consider a subset of indices $J_x\subseteq \{1,\dotsc,T\}$, which depends on the point $x\in [0,1]$. It is important to distinguish between $x$ being an interior point and $x$ being a boundary point of $[0,1]$, since the corresponding kernel averages may exhibit different asymptotic equivalences. Analytically, we can examine the behavior of the kernel average ``near" the boundaries instead of exactly at the boundaries. This approach is particularly convenient when evaluating the boundary bias of kernel estimators \citep[see][among others]{muller,wand_jones}. Motivated by these ideas, we will define the set of indices $J_x$ which will play a role for compactly supported kernels. 

Let $T_0\in\N$ and $L_T$ be a sequence of positive numbers (possibly constant) such that  $L_T\geq1$, $L_T h < 1/2$ and $T L_T h > 1$ for all $T\geq T_0$.  For every $T\geq T_0$, define the index set 
\begin{equation}\label{eqb10}
	J_x(L_T)=\{i\in [T]: i/T\in C_x(L_T)\},\qquad x\in[0,1],
\end{equation}
where $[T] := \{1,2,\dots,T\}$ and 
\begin{equation}\label{eqb11}
	C_x(L_T)=\begin{cases} 
		[0,x+ L_Th],&  \text{if }  x\in[0, L_Th]\\
		[x- L_Th,x+L_Th],&  \text{if }  x\in( L_Th,1-L_Th)\\
		[x- L_Th,1],& \text{if }  x\in[1- L_Th,1]
	\end{cases}.
\end{equation}
The construction above guarantees that $C_x$ and $J_x$ are well defined and nonempty. The requirement that $L_T h<1/2$ ensures that $C_x(L_T) \subseteq [0,1]$ for all $x \in [0,1]$. 
Moreover, since the design points $\{i/T:i=1,\dots,T\}$ are equally spaced, any interval in $[0,1]$ of length larger than $1/T$ contains at least one
design point. Because $C_x(L_T)$ has length at least $L_T h$ for all $x\in[0,1]$,
the condition $T L_T h>1$ guarantees that $J_x(L_T)$ is nonempty for every $x\in[0,1]$.

\begin{lem}\label{l2.1}
	Let $T\geq T_0$ and let $k_{x,T}(L_T)$ be the cardinality of $J_x(L_T)$. Assume that $L_T h\to 0$ and $T L_T h\to\infty$. Then, pointwise in $x\in[0,1]$,  $k_{x,T}(L_T)\overset{a}{\approx}c_x TL_T h$ where $c_x=1+I(x\in (0,1))$. Moreover, $\sup_{x\in[0,1]} k_{x,T}(L_T)\leq 2TL_Th+1$.
\end{lem}

    In particular, Lemma \ref{l2.1} implies that $\sup_x k_{x,T}(L_T) =O(TL_Th)$.  For our purposes, it is convenient to weaken Assumption A.2 by the following version:
    \begin{itemize}
        \item[A.2'] The function $K:\R\to\R$ satisfies $|K(u)|\leq \bar K<\infty$ and $\int_{\R} \abs{K(u)}du\leq \bar \mu<\infty$. There exist positive constants $\Lambda_1,L_1,L_2<\infty$ such that  $K(u)=0$ for $|u|>L_1$, and $|K(u)-K(u')|\leq \Lambda_1 |u-u'|$ for all $u,u'\in [-L_1, L_1]$.
    \end{itemize}
    Assumption A.2' is strictly weaker than A.2 since the Lipschitz property is required only on $[-L_1,L_1]$, in the former, instead of on the whole real line.

\begin{lem}\label{lext}
Assume that $K$ satisfies Assumption A.2' for a fixed $j\geq 0$. Then there exists a positive constant
$C_j<\infty$ such that
\begin{equation*}
    \sup_{u\in\R} |u|^j |K(u)| \leq C_j.
\end{equation*}
\end{lem}

Under Assumption A.2',  $|u|^j |K(u)|\leq L_1^j \bar K$, and Lemma \ref{lext} follows immediately.
 The next lemma establishes a uniform approximation of integrals by finite sums.

\begin{lem}\label{l2}
	Suppose that the kernel function $K$ satisfies Assumption A.2'. Define the composition $f=g\circ K$, where $g$ is Lipschitz continuous on the range of $K$ satisfying $g(0)=0$. Then, for any fixed $j\in\N$ and any $0\leq a\leq b\leq 1$, it holds uniformly in $x\in[0,1]$ for all sufficiently large $T$  that
	\begin{equation}\label{versao1}
		\abs[\bigg]{\frac{1}{T}\sum_{i: i/T\in[a,b]} f\parent[\bigg]{\frac{i/T-x}{h}}\parent[\bigg]{\frac{i/T-x}{h}}^j-\int_{a}^{b} f\parent[\bigg]{\frac{u-x}{h}}\parent[\bigg]{\frac{u-x}{h}}^j du}\leq \frac{C}{T},
	\end{equation}
    and
    \begin{equation}\label{versao2}
		\abs[\bigg]{\frac{1}{T}\sum_{i: i/T\in[a,b]} f\parent[\bigg]{\frac{i/T-x}{h}}\abs[\bigg]{\frac{i/T-x}{h}}^j-\int_{a}^{b} f\parent[\bigg]{\frac{u-x}{h}}\abs[\bigg]{\frac{u-x}{h}}^j du}\leq \frac{C}{T}.
	\end{equation}
\end{lem}

The integral approximation in Lemma \ref{l2} is applied for Lipschitz transformations of $K(u)$, such as $|K(u)|$. 

\begin{lem} \label{l3}
	Let $K$ be a kernel function satisfying Assumption A.2 and let $\delta>0$. Then
	there exist a function $K^*$ and constants $L, \bar K^*,\mu^*$ such that  $\abs{K^*}\leq \bar K^*<\infty$, $\int_{\R}\abs{K^*(u)}du\leq\mu^*<\infty$ and
	\begin{equation*}
		\abs{x_1-x_2}\leq\delta\leq L_1\implies \abs{K(x_1)-K(x_2)}\leq \delta K^*(x_1),\quad \forall x_1,x_2\in\R.
	\end{equation*} 
	In particular, we may take
    \begin{equation}\label{eqker1}
        K^*(x)=\Lambda_1 I(\abs{x}\leq 2L_1).
    \end{equation}
     Define the composition $f^*=g\circ K^*$, where $g$ is Lipschitz continuous on the range of $K^*$, satisfying $g(0)=0$. Then, for $K^*$ defined as in \eqref{eqker1},  for any fixed $j\in\N$ and any $0\leq a\leq b\leq 1$, it holds uniformly in $x\in[0,1]$ for all sufficiently large $T$ that
	\begin{equation}\label{eqaprox}
		\abs[\bigg]{\frac{1}{T}\sum_{i: i/T\in[a,b]} f^*\parent[\bigg]{\frac{i/T-x}{h}}\parent[\bigg]{\frac{i/T-x}{h}}^j-\int_{a}^{b}f^*\parent[\bigg]{\frac{u-x}{h}}\parent[\bigg]{\frac{u-x}{h}}^j du}\leq \frac{C}{T},
	\end{equation}
    and
    \begin{equation}\label{eqaprox2}
		\abs[\bigg]{\frac{1}{T}\sum_{i: i/T\in[a,b]} f^*\parent[\bigg]{\frac{i/T-x}{h}}\abs[\bigg]{\frac{i/T-x}{h}}^j-\int_{a}^{b}f^*\parent[\bigg]{\frac{u-x}{h}}\abs[\bigg]{\frac{u-x}{h}}^j du}\leq \frac{C}{T}.
	\end{equation}
\end{lem}

\begin{coroll}\label{coro2}
Suppose that the conditions of Lemmas~\ref{l2} and~\ref{l3} hold. Let $K^*$ be given by \eqref{eqker1}, and define $f_{L_T}(u)\coloneqq f(u)I(|u|>L_T)$ and  $f_{L_T}^*(u)\coloneqq f^*(u)I(|u|>L_T)$. Then, for any fixed $j\in\N$, it holds
uniformly in $x\in[0,1]$ for all sufficiently large $T$ that
\begin{align}\label{eqcorol1}
\bigg|
\frac{1}{T}\sum_{i=1}^T
f_{L_T}\bigg(\frac{i/T-x}{h}\bigg)
\bigg(\frac{i/T-x}{h}\bigg)^j-\int_0^1
f_{L_T}\bigg(\frac{u-x}{h}\bigg)
\bigg(\frac{u-x}{h}\bigg)^j
du
\bigg|
&\leq \frac{C}{T}, 
\\\label{eqcorol2}
\bigg|
\frac{1}{T}\sum_{i=1}^T
f_{L_T}^*\bigg(\frac{i/T-x}{h}\bigg)
\bigg(\frac{i/T-x}{h}\bigg)^j-\int_{0}^1
f_{L_T}^*\bigg(\frac{u-x}{h}\bigg)
\bigg(\frac{u-x}{h}\bigg)^j
du
\bigg|
&\leq \frac{C}{T}.
\end{align}
\end{coroll}

Consider the quantities related to kernel smoothing estimation (e.g., Nadaraya-Watson and local linear estimators),
\begin{equation*}
      s_{T,j}(x)\coloneqq
  \frac{1}{Th}\sum_{t=1}^{T}
  \left(\frac{t/T-x}{h}\right)^{j}
  K\!\left(\frac{t/T-x}{h}\right),\qquad j\in\N.
\end{equation*}
The following lemma extends Proposition 1 in \cite{fernandez}.
\begin{lem}\label{l5}
	Under Assumption A.2 with $L_1=1$, for any fixed $j\in\N$, it holds that
	\begin{equation*}
		\sup_{x\in[0,1]}\big|s_{T,j}(x)-\mu_j(x)\big|=O\bigg(\frac{1}{Th}\bigg), 
	\end{equation*}
	 where $\mu_j(x)=\int_{G_x} u^jK(u)du$ with 
	\begin{equation}\label{eqsetg}
		G_x=\begin{cases}
	[-1,0],& \text{if } x= 1\\
	[-1,1],          & \text{if } x\in(0,1)\\
	[0,1], & \text{if  } x=0
\end{cases}.
\end{equation}
\end{lem}

 Lemma \ref{l5} implies that the matrix  $S_{T,x}$ defined in \eqref{eq13_} converges to $S_x$ as $T\to\infty$, where 
\begin{equation}\label{eqb13}
	S_x=\int_{G_x} \left[\begin{array}{cc} 1& u\\
		u&u^2
	\end{array}\right] K(u)du.
\end{equation} 

\begin{lem}\label{l6}
Let $K\ge0$ satisfy Assumption A.2 with $L=1$, and define $S_x$ as in \eqref{eqb13}.
If $\mu(\{u\in G_x:K(u)>0\})>0$, then $S_x$ is positive definite. Moreover, there exist
$\lambda_0>0$ and $T_0<\infty$ such that
\begin{equation}\label{eqlmin}
\min_{x\in[0,1]}\{\lambda_{\min}(S_{T,x}) \}\geq\lambda_0,\qquad \forall\,T\ge T_0,
\end{equation}
where $\lambda_{\min}(S_{T,x})$ denotes the smallest eigenvalue of $S_{T,x}$.
\end{lem}

Throughout the next lemma, $\|\cdot\|$ stands for the Euclidean ($L_2$) norm in $\mathbb{R}^2$.

\begin{lem}\label{l7}
Assume the conditions of Lemma \ref{l6} hold. Then, for all $v\in\R^2$ and all $T$ sufficiently large,
\begin{equation*}
    \sup_{x\in[0,1]}\|S_{T,x}^{-1}v\|\leq\frac{1}{\lambda_0}\|v\|.
\end{equation*}
\end{lem}

    To prove the uniform convergence rates in Theorems \ref{teo2} and \ref{teo3}, two intermediate results are used.  The first provides a  tail probability bound for partial sums under $\alpha$-mixing dependence, while the second establishes a variance bound for blocks of kernel averages.

    We first state an exponential-type inequality for triangular arrays, based on Theorem 2.1 of \cite{liebscher}, which in turn follows from Theorem 5 in \cite{rio}.
	
	\begin{lem}[Liebscher-Rio]\label{l1}
		Let $\{Z_{i,T}\}$ be a zero-mean triangular array such that $\abs{Z_{i,T}}\leq b_T$, and let $\alpha_T$ denote its strong mixing coefficients. Then, for any $\epsilon>0$ and $1\leq m_T\leq T$ satisfying $4b_Tm_T<\epsilon$,
        \begin{equation*}
            P\parent[\Bigg]{\abs[\Bigg]{\sum_{i=1}^T Z_{i,T}}>\epsilon}\leq 4\exp\parent[\bigg]{-\frac{\epsilon^2}{64\sigma^2_{T, m_T}T/{ m_T}+\tfrac{8}{3}\epsilon b_T  m_T}}+4\alpha_T( m_T)\frac{T}{ m_T},
        \end{equation*}
		where $\sigma^2_{T, m_T}=\sup_{0\leq j \leq T-1}E[(\sum_{i=j+1}^{\min\{j+m_T,T\}} Z_{i,T})^2]$.
	\end{lem}
    
	To bound the variance term $\sigma^2_{T, m_T}$ appearing in Lemma \ref{l1}, one must account for the possibly nonzero covariances of $\{Z_{t,T}\}$. The approach of \cite{hansen} and \cite{kristensen} bounds such covariances separately over short, medium, and long lags. In our fixed-design setting, however, such a partition is unnecessary and can be avoided altogether by exploiting the uniform integral approximation established in Lemma \ref{l2}.  We emphasize that throughout the proofs of
Theorems \ref{teo2}-\ref{teo3}, the variance term $\sigma^2_{T,m_T}$ appearing
in Lemma \ref{l1} will be understood as the block variance of the
array $\{Z_{t,T}\}$ restricted to an index set induced by the kernel. To be precise, for any  positive increasing sequence $L_T$, define for each $x\in[0,1]$ the index set
\begin{equation}
   J_{l}(L_T)\coloneqq\Big\{i\in\{1,\dots,T\}:\Big|\frac{i/T-x}{h}\Big|\le L_T\Big\},
\qquad n_T(x)\coloneqq \#J_{x}(L_T).
\end{equation}
Let $\{i_r(x)\}_{r=1}^{n_T(x)}$ be the increasing enumeration of $J_{x}(L_T)$. Given a triangular array $\{Y_{i,T}(\gamma):T\geq1, 1\leq i\leq T\}$, set
\begin{equation*}
    \widetilde Y_{r,T}(x,\gamma)\coloneqq Y_{i_r(x),T}(\gamma),\qquad r\in\{1,\dots,n_T(x)\},
    \end{equation*}
    and
    \begin{equation*}
    \sigma^2_{n_T(x), m_T}=\sup_{0\leq \ell\leq n_T(x)-1}\var\bigg(\sum_{r=\ell+1}^{\min\{\ell+m_T,n_T(x)\}}\widetilde Y_{r,T}(\gamma)K\bigg(\frac{i_r(x)/T-x}{h}\bigg) \bigg(\frac{i_r(x)/T-x}{h}\bigg)^j\bigg).
\end{equation*}
The quantity $\sigma^2_{n_T(x),m_T}$ defined above corresponds to the variance
term that enters the application of Lemma~\ref{l1} in the proofs of
Theorems \ref{teo2} and \ref{teo3}.

	
	\begin{lem}\label{teo1}
		Let $\{Y_{i,T}(\gamma):T\geq1, 1\leq i\leq T\}$ be an $\alpha$-mixing triangular array satisfying A.1 and the moment condition in A.3\eqref{eqda1}.  Assume $\beta>s/(s-2)$ and let $m_T$ be a positive sequence such that $m_T\leq n_T(x)$ for all $x\in[0,1]$.  Under A.2, for all sufficiently large $T$, we have 
        \begin{equation*}			
        \sup_{x\in[0,1]}\sigma^2_{n_T(x), m_T}\leq C(1+\|\gamma\|^\lambda)^{2/s}m_T.
		\end{equation*}
	\end{lem}

    \end{document}


\pagestyle{myheadings}
\markboth{}{Matsuoka and Torrent}

\thispagestyle{empty}
{\centering
\Large{\bf Supplementary material for ``Uniform convergence of kernel averages under fixed design with heterogeneous dependent data".} \vspace{.5cm}\\
\normalsize{ {\bf 
Danilo H. Matsuoka${}^{\mathrm{a}}$,\let\thefootnote\relax\footnote{\hskip-.3cm$\phantom{s}^\mathrm{a}$Research Group of Applied Microeconomics - Department of Economics, Federal University of Rio Grande.} Hudson da Silva Torrent${}^\mathrm{b}$\let\thefootnote\relax\footnote{\hskip-.3cm$\phantom{s}^\mathrm{b}$Mathematics and Statistics Institute - Universidade Federal do Rio Grande do Sul.
}
 \\
\let\thefootnote\relax\footnote{E-mails: danilomatsuoka@gmail.com (Matsuoka); hudsontorrent@gmail.com (Torrent)}
\vskip.3
cm}}
}
\section*{Introduction}

This supplementary material contains the proofs of auxiliary lemmas and technical results omitted from the main text for brevity. All references in square brackets (e.g., [Section 2], [Appendix B]) refer  to the corresponding parts of the main paper.

In what follows, we use the following notations: $[k] := \{1,2,\dots,k\}$ for any $k\in\N$, $\sup_x\coloneqq \sup_{x\in [0,1]}$, $\sup_\gamma\coloneqq \sup_{\gamma\in\Theta_T}$, $\sup_i\coloneqq\sup_{1\leq i\leq T}$, $\sup_T\coloneqq \sup_{T\geq 1}$, $\sup_{T,i}\coloneqq \sup_T\sup_i$  and $\sup_{x,\gamma}\coloneqq \sup_x\sup_\gamma$.
Throughout, $C>0$ denotes a generic constant that can take different values at different occurrences and is independent of $T, x$, and $\gamma$. The notation ``$\overset{a}{\approx}$'' denotes asymptotic equivalence.

\section*{Proofs of auxiliary results}

Throughout this document, references in square brackets (e.g., [Section 1], [Appendix B]) refer to the corresponding parts of the main paper.

\noindent\textbf{Proof of Lemma 1}.
	Define $h_{1,x}=\max\{0,x-L_Th\}$ and $h_{2,x}=\min\{1,x+L_Th\}$. Then $$J_x(L_T)=\{i\in[T]: i/T\in[x-L_Th,x+L_Th]\}=\{i\in[T]: i/T\in[h_{1,x},h_{2,x}]\},$$ and $k_{x,T}(L_T)=\# J_x(L_T)$ is the number of consecutive integers $i$ satisfying $h_{1,x}\leq i/T \leq h_{2,x}$. On one hand, $\ceil{Th_{1,x}}$ is the smallest integer $i$ such that $i\geq Th_{1,x}$. On the other, $\floor{Th_{2,x}}$ is the largest integer $i$ such that $i\leq Th_{2,x}$. Hence, $k_{x,T}(L_T)=\floor{Th_{2,x}}-\ceil{Th_{1,x}}+1$. Then
    \begin{align}\label{eqas1}
k_{x,T}(L_T)
&\leq (Th_{2,x}) - (Th_{1,x}) + 1
= T(h_{2,x}-h_{1,x}) + 1, \\\label{eqas2}
k_{x,T}(L_T)
&\geq (Th_{2,x} - 1) - (Th_{1,x} + 1) + 1
= T(h_{2,x}-h_{1,x}) - 1.
\end{align}
	Fix $x\in[0,1]$. If $x\in\{0,1\}$, then $h_{2,x}-h_{1,x}=L_T h$ for all $T$. If $x\in(0,1)$, then  $L_Th<\min\{x,1-x\}$ for all $T$ sufficiently large, which implies $h_{2,x}-h_{1,x}=(x+L_Th)-(x-L_Th)=2L_Th$.
    Dividing \eqref{eqas1}-\eqref{eqas2} by $TL_Th$ yields
    \begin{equation}\label{eqas3}
        \frac{h_{2,x}-h_{1,x}}{L_Th}-\frac{1}{TL_Th}\leq \frac{k_{x,T}(L_T)}{TL_Th}\leq \frac{h_{2,x}-h_{1,x}}{L_Th}+\frac{1}{TL_Th},
    \end{equation}
    for all $T$ large enough. Taking limits in \eqref{eqas3} gives $k_{x,T}(L_T)/(TL_Th)\to c_x$, $T\to\infty$, where $c_x=1+I(x\in(0,1))$, since $1/(TL_Th)=o(1)$. Finally, since $\sup_x\{h_{2,x}-h_{1,x}\}\leq 2L_Th$, inequality \eqref{eqas1} implies $\sup_x k_{x,T}(L_T)\leq 2T L_Th+1$.
    \qed

 \noindent\textbf{Proof of Lemma 3}. The proof of inequality (76) [Appendix B] follows from the same arguments as those for (75) [Appendix B], so we only prove the latter. Fix $0\leq a\leq b\leq 1$ and $x\in[0,1]$. Define $\zeta(u)\coloneqq f(u)u^{j}$ and $\ell_x(u)\coloneqq \zeta (\frac{u-x}{h})$, for all $u\in[0,1]$. With this notation, 
    \begin{equation*}
        \ell_x(i/T)=
f\bigg(\frac{i/T-x}{h}\bigg)
\bigg(\frac{i/T-x}{h}\bigg)^j.
    \end{equation*}
    Let $i_a\coloneqq \lceil aT\rceil$ and $i_b\coloneqq \lfloor bT\rfloor$. Then
\begin{equation*}
    \frac{1}{T}\sum_{i: i/T\in[a,b]}\ell_x(i/T)
= \frac{1}{T}\sum_{i=i_a}^{i_b}\ell_x(i/T)=\frac{1}{T}\ell_x(i_a/T)+\frac{1}{T}\sum_{i=i_a+1}^{i_b}\ell_x(i/T),
\end{equation*}
    and
    \begin{equation*}
\int_a^b \ell_x(u)du=
\int_{i_a/T}^{i_b/T}\ell_x(u)du
+\int_a^{i_a/T}\ell_x(u)du
+\int_{i_b/T}^{b}\ell_x(u)\,du.
    \end{equation*}
    Hence, by the triangle inequality,
    \begin{align*}
        \bigg|
\frac{1}{T}\sum_{i=i_a}^{i_b}\ell_x(i/T)-\int_a^b \ell_x(u)du
\bigg|&\leq \bigg|
\frac{1}{T}\sum_{i=i_a+1}^{i_b}\ell_x(i/T)-\int_{i_a/T}^{i_b/T}\ell_x(u)du\bigg|+\frac{1}{T}|\ell_x(i_a/T)|\\
&\qquad +\bigg|
\int_a^{i_a/T}\ell_x(u)du
\bigg|+\bigg|
\int_{i_b/T}^{b}\ell_x(u)\,du
\bigg|\\
&\coloneqq A_{x,T}+ B_{x,T}+ C_{x,T}+D_{x,T}.
    \end{align*}
In addition, from the definitions of $i_a$ and $i_b$, we must have
\begin{equation}\label{eql2}
    0 \leq \frac{i_a}{T} - a \leq \frac{1}{T}, 
\quad \text{ and } \quad
0 \leq b- \frac{i_b}{T} \leq \frac{1}{T}.
\end{equation}
    Since $K$ is Lipschitz on $[-L_1,L_1]$ and $g$ is Lipschitz on the range of $K$, the composition $f=g\circ K$ is Lipschitz on $[-L_1,L_1]$. Moreover,  because $K(u)=0$ for $|u|>L_1$ and $g(0)=0$, we have $f(u)=0$ for $|u|>L_1$. Hence,  $f$ has compact support, and thus is bounded and integrable on $\R$.  Therefore,  $f$ satisfies the conditions of Lemma 2. Applying this lemma yields $\sup_{u\in\R}|u|^j|f(u)|\leq C_j$, and so $|\ell_x(u)|\leq\sup_{u\in\R}|u|^j|f(u)|\leq C_j$. Therefore, using \eqref{eql2},
    \begin{align}\label{eqB}
    \sup_{x\in[0,1]} B_{x,T}&\leq \sup_{u\in\R} |\ell_x(u)|\frac{1}{T}\leq \frac{C_j}{T},\\\label{eqC}
        \sup_{x\in[0,1]} C_{x,T}&\leq \sup_{u\in\R} |\ell_x(u)|\int_{a}^{i_a/T}du\leq C_j\bigg(\frac{i_a}{T}-a\bigg)\leq \frac{C_j}{T},\\\label{eqD}
        \sup_{x\in[0,1]} D_{x,T}&\leq \sup_{u\in\R} |\ell_x(u)|\int_{i_b/T}^b du\leq C_j\bigg(b-\frac{i_b}{T}\bigg)\leq \frac{C_j}{T}.
    \end{align}
    Next, we focus on $A_{x,T}$. Since $u\mapsto u^{j}$ is continuously differentiable on $[-L_1,L_1]$, it is Lipschitz on this interval. Moreover, $f$ is bounded and Lipschitz on $[-L_1,L_1]$. Therefore, the product $\zeta(u)= f(u)u^{j}$, is Lipschitz on $[-L_1,L_1]$. Indeed,
    \begin{equation}\label{eqlips}
    |\zeta(u)-\zeta(v)|\leq |f(u)-f(v)||u|^j+|f(v)||u^j-v^j|\leq C|u-v|,\qquad u,v\in[-L_1,L_1]. 
     \end{equation}
     Thus, for all $u,v\in\R$ such that $|(u-x)/h|\leq L_1$ and $|(v-x)/h|\leq L_1$,
      \begin{equation} \label{eqlipsoi2}
    |\ell_x(u)-\ell_x(v)|=\bigg|\zeta\bigg(\frac{u-x}{h}\bigg)-\zeta\bigg(\frac{v-x}{h}\bigg)\bigg|\leq \frac{C}{h}|u-v|.
     \end{equation} 
     Because $f(u)=0$ for $|u|>L_1$, we also have $\zeta(u)=f(u)u^{j}=0$ for  $|u|>L_1$, so 
     \begin{equation}\label{eqzero}
         \ell_x(u)=\zeta((u-x)/h)=0,\ \text{ if }|(u-x)/h|>L_1.
     \end{equation}   
 Let $\underline J_{x,T}=\{i\in\{i_a+1,\dotsc,i_b\} :i/T \in C_x(L_1)\}$ with $C_x(L_1)$ as in equation (71) [Appendix B].  The sum appearing in $A_{x,T}$ over $\{i_a+1,\dots,i_b\}$ is first restricted to
$\underline J_{x,T}$, since $\ell_x(i/T)=0$ otherwise, according to \eqref{eqzero}. We partition $\underline J_{x,T}=J_{x,T}^*\cup \partial J_{x,T}^*$ where $J_{x,T}^*\coloneqq \{i\in\{i_a+1,\dotsc,i_b\} :[(i-1)/T,i/T] \subseteq C_x(L_1)\}$ and $\partial J_{x,T}^*\coloneqq \underline J_{x,T}\setminus J_{x,T}^*$. Since $J_{x,T}^*\subseteq J_x(L_1)=\{i\in[T:i/T\in C_x(L_1)]\}$, Lemma 1 implies that $\sup_{x\in[0,1]}\# J_{x,T}^*=O(Th)$. Note that $\partial J_{x,T}^*$ is empty if $\underline J_{x,T}=\Emptyset$. If  $\underline J_{x,T}\neq \Emptyset$ we have that  $\partial J_{x,T}^*$ is empty if and only if $\min \underline J_{x,T}/T=\min C_x(L_1)$. Otherwise, $\partial J_{x,T}^*=\{i_0\}$ and hence $\#\partial J_{x,T}^*=1$. Thus, $\sup_{x\in[0,1]}\#\partial J_{x,T}^*=1$.

By \eqref{eqlipsoi2}, $\ell_x(u)$ is Lipschitz on $u\in[(i-1)/T,i/T]$ if $i\in J_{x,T}^*$.  Thus, by the Mean Value Theorem for integrals, for each  $i\in J_{x,T}^*$, there exists $\xi_i \in ((i-1)/T,i/T)$ such that
\begin{equation*}
    \int_{(i-1)/T}^{i/T} \ell_x(u)\,du = \frac{1}{T} \ell_x(\xi_i).
\end{equation*}
Therefore, from the fact that $|i/T - \xi_i| \leq 1/T$,  we have for all $i\in J_{x,T}^*$
\begin{align*}
    \bigg|\frac{1}{T} \ell_x(i/T) - \int_{(i-1)/T}^{i/T} \ell_x(u) \,du \bigg|&= \frac{1}{T}\big| \ell_x(i/T) - \ell_x(\xi_i) \big|\leq  \frac{C}{Th} \big|\tfrac{i}{T}-\xi_i\big|\leq \frac{C}{T^2h}.
\end{align*}
In addition, since $\sup_{u\in\R}|\ell_x(u)|\leq C_j\leq C$, for all $i\in \partial J_{x,T}^*$
\begin{equation*}
    \bigg|\frac{1}{T} \ell_x(i/T) - \int_{(i-1)/T}^{i/T} \ell_x(u) du\bigg|\leq \bigg|\frac{1}{T} \ell_x(i/T)\bigg| + \bigg|\int_{(i-1)/T}^{i/T} \ell_x(u) du\bigg|\leq \frac{C}{T}.
\end{equation*}
 Thus, using the above results, it follows that for all sufficiently large $T$ 
    \begin{align}\notag
         A_{x,T}&\leq\sum_{i\in J_{x,T}^*}\bigg|\frac{1}{T} \ell_x(i/T) - \int_{(i-1)/T}^{i/T} \ell_x(u) du\bigg|+\sum_{i\in \partial J_{x,T}^*}\bigg|\frac{1}{T} \ell_x(i/T) - \int_{(i-1)/T}^{i/T} \ell_x(u) du\bigg|\\\label{eqA}
         &\leq \sup_{x\in[0,1]}\#J_{x,T}^*\frac{C}{T^2h}+\sup_{x\in[0,1]}\#\partial J_{x,T}^*\frac{C}{T} \leq \frac{C}{T}.
    \end{align}
    From the bounds \eqref{eqB}-\eqref{eqD} and \eqref{eqA}, we have   $\sup_{x\in[0,1]}\big|\tfrac{1}{T}\sum_{i:i/T\in[a,b]} \ell_x(i/T)-\int_a^b\ell_x(u)du\big|=O(1/T)$ as claimed.
\qed

\noindent\textbf{Proof of Lemma 4}. We begin showing that the function defined in (85) [Appendix B] satisfies the conditions on $K^*$. Define $K^*(x)=\Lambda_1 I(\abs{x}\leq 2L_1)$. 
Fix $\delta>0$ and let $x_1,x_2$ satisfy $\abs{x_1-x_2}\leq\delta\leq L_1$. 
Since $K$ is Lipschitz continuous, we have
\begin{equation*}
    \abs{K(x_1)-K(x_2)}\leq \Lambda_1\abs{x_1-x_2}
=\Lambda_1\abs{x_1-x_2}\{I(\abs{x_1}\leq 2L_1)+I(\abs{x_1}>2L_1)\}.
\end{equation*}
If $|x_1|>2L_1$ and $|x_1-x_2|\leq L_1$, then $|x_2|\geq |x_1|-|x_1-x_2|>2L_1-L_1=L_1$.
Since $K$ has compact support, $K(x_1)-K(x_2)=0$, so the term $I(\abs{x_1}>2L_1)$ contributes nothing. 
Hence,  it holds that
$\abs{K(x_1)-K(x_2)}\leq \delta K^*(x_1)$. Moreover,  
$\abs{K^*}\leq \Lambda_1$, and 
\(\int_{\R} \abs{K^*(u)}du\leq \Lambda_1(4L_1)\).

	We now turn to the integral approximation in (86) [Appendix B]. Let $ K^*(x)=\Lambda_1 I(\abs{x}\leq 2L_1)$. Then $K^*$ is bounded, integrable and Lipschitz on $[-2L_1,2L_1]$. Hence, the application of Lemma 3 gives the result. 
\qed

\noindent\textbf{Proof of Corollary 1}. Define, for each $x\in[0,1]$, $C_x^c(L_T)\coloneqq [0,1]\setminus C_x(L_T)$, where $C_x(L_T)$ is given in (71) [Appendix B]. Then
\begin{align*}
    u\in C_x(L_T)
\iff \Big|\frac{u-x}{h}\Big|\leq L_T,
\end{align*}
and therefore
\begin{equation}\label{eqdocoro2}
    I \Big(\Big|\frac{u-x}{h}\Big|>L_T\Big)=I \big(u\in C_x^c(L_T)\big), \qquad u\in[0,1].
\end{equation}
Let $\ell_x(u)\coloneqq f((u-x)/h)((u-x)/h)^j$.
Using \eqref{eqdocoro2}, we rewrite
\begin{align*}
\frac{1}{T}\sum_{i=1}^T f_{L_T}\bigg(\frac{i/T-x}{h}\bigg)\bigg(\frac{i/T-x}{h}\bigg)^j&= \frac{1}{T}\sum_{i=1}^T \ell_x(i/T) I\Big(\Big|\frac{i/T-x}{h}\Big|>L_T\Big)\\
&= \frac{1}{T}\sum_{i: i/T\in C_x^c(L_T)}\ell_x(i/T), \\[2mm]
\int_0^1 f_{L_T}\bigg(\frac{u-x}{h}\bigg)\bigg(\frac{u-x}{h}\bigg)^j du &=
\int_0^1 \ell_x(u) I\Big(\Big|\frac{u-x}{h}\Big|>L_T\Big) du\\
&= \int_{C_x^c(L_T)} \ell_x(u)du.
\end{align*}
Hence it suffices to bound $\sup_{x\in[0,1]}|\frac{1}{T}\sum_{i:\, i/T\in C_x^c(L_T)}\ell_x(i/T)
-\int_{C_x^c(L_T)} \ell_x(u)du|$.

Next, we characterize $C_x^c(L_T)$. From (71) [Appendix B], the complement $C_x^c(L_T)$
is the union of at most two subintervals of $[0,1]$. More precisely,
\begin{equation*}
C_x^c(L_T)=
\begin{cases}
(x+L_T h,1], & x\in[0,L_T h],\\
[0,x-L_T h) \cup (x+L_T h,\,1], & x\in(L_T h,1-L_T h),\\
[0,x-L_T h), & x\in[1-L_T h,1].
\end{cases}
\end{equation*}
Define  $I_{1,x,T}\coloneqq [0,\max\{0,x-L_T h\})$ and $I_{2,x,T}\coloneqq (\min\{1,x+L_T h\},1]$. Note that $[0,0)=\Emptyset$ and $(1,1]=\Emptyset$, so $I_{1,x,T}$ or $I_{2,x,T}$ may be empty when $x\in[0,L_T h]\cup[1-L_T h,1]$. Then $C_x^c(L_T)= I_{1,x,T}\cup I_{2,x,T}$. Hence, by the triangle inequality, 
\begin{align*}
\bigg|\frac{1}{T}\sum_{i: i/T\in C_x^c(L_T)}\ell_x(i/T)
-\int_{C_x^c(L_T)} \ell_x(u) du\bigg|
&\leq
\sum_{k=1}^2
\bigg|
\frac{1}{T}\sum_{i: i/T\in I_{k,x,T}}\ell_x(i/T)
-\int_{I_{k,x,T}} \ell_x(u) du
\bigg|.
\end{align*}
For each $k\in\{1,2\}$, if $I_{k,x,T}=\emptyset$ then the corresponding term is zero.
Otherwise, $I_{k,x,T}$ is a half-open interval contained in $[0,1]$. Replacing $I_{k,x,T}$ by its closure $\bar I_{k,x,T}$ adds at most one term to the finite sum, corresponding to a (possible) boundary point  $i/T\in\{x\pm L_Th\}$, and does not affect the integral. Therefore, since $\sup_{u\in\R}|\ell_x(u)|\leq C$, this replacement can generate only an error of order $O(1/T)$ uniformly in $x\in[0,1]$:
\begin{equation*}
   \bigg| \frac{1}{T}\sum_{i: i/T\in I_{k,x,T}}\ell_x(i/T)
-\frac{1}{T}\sum_{i: i/T\in \bar I_{k,x,T}}\ell_x(i/T)
\bigg|\leq\frac{1}{T}\sup_{u\in\R}|\ell_x(u)|\leq \frac{C}{T}.
\end{equation*}
Moreover, each $\bar I_{k,x,T}$ is of the form $[a,b]$ whenever $I_{k,x,T}\neq\emptyset$. It follows from Lemma 3 that
\begin{equation*}
   \sup_{x\in[0,1]} \bigg|
\frac{1}{T}\sum_{i: i/T\in I_{k,x,T}}\ell_x(i/T)
-\int_{I_{k,x,T}} \ell_x(u) du
\bigg|\leq \frac{C}{T}, \qquad k\in\{1,2\},
\end{equation*}
for all sufficiently large $T$. Hence
\begin{equation*}
    \sup_{x\in[0,1]}\bigg|\frac{1}{T}\sum_{i: i/T\in C_x^c(L_T)}\ell_x(i/T)
-\int_{C_x^c(L_T)} \ell_x(u) du\bigg|\leq \frac{C}{T},
\end{equation*}
which proves (88) [Appendix B]. The second inequality follows by the same argument, replacing $f$ with $f^*$ and
using Lemma 4, so is omitted.
\qed

\noindent\textbf{Proof of Lemma 5}. Apply Lemma 3 with $f(u)=K(u)u^j$ and $[a,b]=[0,1]$, together with the change of variables $w=(u-x)/h$ to obtain uniformly in $x\in[0,1]$
\begin{align*}
    s_{T,j}(x)&=	\int_0^1 K\bigg(\frac{u-x}{h}\bigg)\bigg(\frac{u-x}{h}\bigg )^jdu+O\bigg(\frac{1}{Th}\bigg)\\
    &=\int_{-x/h}^{(1-x)/h}K(w)w^jdw +O\bigg(\frac{1}{Th}\bigg).
\end{align*}
Since $K(w)=0$ for $|w|>1$, the last integral equals $\mu_j(x)$ for all sufficiently large $T$.
\qed

\noindent\textbf{Proof of Lemma 6}. For all nonzero vector $y=(y_1,y_2)'\in\mathbb{R}^2$, it holds that
\begin{equation}\label{eqpro}
    y^\intercal S_xy=\int_{G_x} y^\intercal\left[\begin{array}{cc} 1& u\\
		u&u^2
	\end{array}\right]y K(u)\,d\mu(u)=\int_{G_x} (y_1+y_2 u)^2 K(u)\,d\mu(u)\geq 0.
\end{equation}
Since $\mu(\{K>0\}\cap G_x)>0$ and $(y_1+y_2 u)^2=0$ at most at one point of $\{K>0\}\cap G_x$, there exists a subset of $G_x$ with positive measure on which $K(u)>0$ and $(y_1+y_2 u)^2>0$. Therefore, the integral in \eqref{eqpro} must be strictly positive, and $S_x$ is positive definite.

Next, we provide a uniform lower bound for $\lambda_{\min} (S_{T,x})$. Since $S_x$ is positive definite for each $x\in[0,1]$, both $\det(S_x)$ and $\tr(S_x)$ are strictly positive and continuous functions of $x$. As $[0,1]$ is compact, there exist 
\begin{equation*}
    \ubar d\coloneqq\min_{x\in[0,1]} \det(S_x)>0, 
\quad 
\bar t\coloneqq\max_{x\in[0,1]} \tr(S_x)<\infty \quad\text{and}\quad \ubar t\coloneqq\min_{x\in[0,1]} \tr(S_x).
\end{equation*}
By Lemma 5, $S_{T,x}\to S_x$ uniformly in $x\in[0,1]$. From the continuity of $\det(\cdot)$ and $\tr(\cdot)$, for any $\epsilon\in(0,1)$ 
there exists $T_0<\infty$ such that, for all $T\ge T_0$ and all $x\in[0,1]$,
\begin{equation}\label{eqdets}
    \det(S_{T,x})\geq (1-\epsilon)\ubar d,
\quad 
\tr(S_{T,x}) \leq (1+\epsilon)\bar t \quad\text{and}\quad \tr(S_{T,x}) \geq (1-\epsilon)\ubar t
\end{equation}
Since $S_{T,x}$ is a symmetric, its eigenvalues satisfy $\lambda_{\min}(S_{T,x})\lambda_{\max}(S_{T,x})=\det(S_{T,x})$ and $\lambda_{\min}(S_{T,x})+\lambda_{\max}(S_{T,x})=\tr(S_{T,x})$, where $\lambda_{\max}(S_{T,x})$ denotes the largest eigenvalue of $S_{T,x}$. From \eqref{eqdets}, $\det(S_{T,x})>0$ and $\tr(S_{T,x})>0$, and for all $T\ge T_0$ and all $x\in[0,1]$,
\begin{equation*}
    \lambda_{\min}(S_{T,x}) \geq \frac{\det(S_{T,x})}{\tr(S_{T,x})}\geq \frac{1-\epsilon}{1+\epsilon}\frac{\ubar d}{\bar t}\coloneqq\lambda_0>0.
\end{equation*}
Since $\lambda_0$ does not depend on $x$, the bound holds uniformly over $[0,1]$.
\qed

\noindent\textbf{Proof of Lemma 7}. By Lemma 6, there exists $T_0<\infty$ and $\lambda_0>0$ such that 
$\lambda_{\min}(S_{T,x})\geq\lambda_0$ for all $x\in[0,1]$ and all $T\geq T_0$. 
Since each $S_{T,x}$ is a real symmetric positive definite matrix, it admits an orthogonal diagonalization
$S_{T,x}=PDP^\intercal$, where $P=(p_1,p_2)$ is an orthogonal matrix whose columns are eigenvectors of $S_{T,x}$ and $D=\diag(\lambda_1,\lambda_2)$ contains its eigenvalues. 
Then, for any $v\in\mathbb{R}^2$,
\begin{equation*}
 	\|S_{T,x}^{-1}v\|^2
= v^\intercal(S_{T,x}^{-1})^2v
= (P^\intercal v)^\intercal(D^{-1})^2(P^\intercal v)
= \sum_{i=1}^2 \frac{(p_i^\intercal v)^2}{\lambda_i^2}
\leq \frac{1}{\lambda_0^2}\sum_{i=1}^2 (p_i^\intercal v)^2
= \frac{1}{\lambda_0^2}\|v\|^2,
\end{equation*}
where the inequality follows from $\lambda_i\ge\lambda_0$ for all $i$. 
Taking the supremum over $x\in[0,1]$ yields the desired result. \qed

 \textbf{Proof of Lemma 9.} Fix $x\in[0,1]$, $\gamma\in\Theta$ and $0\leq \ell\leq T-1$. Define $k_{r,T}(x)\coloneqq K((i_r(x)/T-x)/h)$ and $\pi_{r,j,T}(x)\coloneqq((i_r(x)/T-x)/h)^j$, and let
    \begin{equation*}
    S_{\ell,m_T}(x,\gamma)\coloneqq    \var\bigg(\sum_{r=\ell+1}^{\min\{\ell+m_T,n_T(x)\}}\widetilde Y_{r,T}(\gamma)\pi_{r,j,T}(x)k_{r,T}(x)\bigg).
    \end{equation*}
    If the original array  $\{Y_{i,T}(\gamma):T\geq1, 1\leq i\leq T\}$ satisfies Assumption A.1, then the subarray $\{\widetilde Y_{r,T}(\gamma):T\geq 1, 1\leq r\leq n_T(x)\}$ is also $\alpha$-mixing with mixing coefficients satisfying the same bounds, $\tilde \alpha_{\gamma,T}(k)\leq  \alpha_{\gamma,T}(k)\leq A k^{-\beta}$. 
    Moreover, under Assumption A.3(4), the uniform moment bound is inherited
    \begin{equation*}
        \sup_{T\geq 1}\,\sup_{1\leq r\leq n_T(x)}\E(| \widetilde Y_{r,T}(\gamma)|^s)\leq  \sup_{T\geq 1}\,\sup_{1\leq i\leq T}\E(| Y_{i,T}(\gamma)|^s)\leq \bar C_1(1+\|\gamma\|^\lambda).
    \end{equation*}
    Hence, Assumptions A.1 and A.3(4) continue to hold for the subarray $\{\widetilde Y_{r,T}(\gamma)\}$.  Therefore, using Davydov's \citep[Theorem 1.1 in][]{bosq} and H\"older's inequalities, and Lemma 2, we obtain that, for some constants $\beta,s>2$, 
	\begin{align}\notag
		 S_{\ell,m_T}&(x,\gamma)\leq \sum_{r,t=\ell+1}^{\min\{\ell+m_T,n_T(x)\}}\abs[\big]{ k_{r,T}(x)k_{t,T}(x)\pi_{r,j,T}(x)\pi_{t,j,T}(x)\cov(\widetilde Y_{r,T},\widetilde Y_{t,T})}\\\notag
         &\leq \Big(\sup_{u\in\R}|K(u)| |u|^j\Big)^2\sum_{r,t=\ell+1}^{\ell+m_T}\abs[\big]{\cov(\widetilde Y_{r,T},\widetilde Y_{t,T})}\\\notag
		&\leq C_j^2\bigg\{\sum_{r=\ell+1}^{\ell+m_T} \E(|\widetilde Y_{r,T}|^2) +\sum_{\substack{r,t=\ell+1\\ r\neq t}}^{\ell+m_T}\abs[\big]{\cov(\widetilde Y_{r,T},\widetilde Y_{t,T})}\bigg\}\\\notag
       & \leq C_j^2 \bigg\{\sum_{r=\ell+1}^{\ell+m_T} (\E|\widetilde Y_{r,T}|^s)^{2/s}+ \sum_{\substack{r,t=\ell+1\\ r\neq t}}^{\ell+m_T}C\alpha_T(\abs{r-t})^{(s-2)/s}(\E|\widetilde Y_{r,T}|^s)^{1/s}(\E|\widetilde Y_{t,T}|^s)^{1/s}\bigg\}\\\label{eqparci}
       &\leq C(1+\|\gamma\|^\lambda)^{2/s}\bigg\{m_T+ \sum_{r=\ell+1}^{\ell+m_T} \bigg(\sum_{\substack{t=\ell+1\\ t\neq r}}^{\ell+m_T} \abs{r-t}^{-\beta(s-2)/s}\bigg)\bigg\}.
	\end{align}
    Note that, for each $r\in\{\ell+1,\dotsc,\ell+m_T\}$, there exists $C_s>0$ depending only on $s$ such that
    \begin{align*}
        \sum_{\substack{t=\ell+1\\ t\neq r}}^{\ell+m_T} \abs{r-t}^{-\beta(s-2)/s}\leq 2\sum_{i=1}^{m_T}i^{-\beta(s-2)/s}\leq 2\sum_{i=1}^{\infty}i^{-\beta(s-2)/s}\leq  2C_s,
    \end{align*}
since $\beta(s-2)/s>1$.
 Thus, from \eqref{eqparci}, we conclude that $$\sup_{x}\sigma^2_{n_T(x),m_T}=\sup_x\sup_{0\leq \ell\leq n_T(x)-1} S_{\ell,m_T}(x,\gamma)\leq C(1+\|\gamma\|^\lambda)^{2/s}m_T,$$ as claimed.
	\qed

\FloatBarrier
\bibliographystyle{apalike}
\bibliography{bib}